\documentclass[12pt,a4paper,reqno]{amsart}
\usepackage[left=3truecm,right=2.5truecm,top=2.5truecm,bottom=2.5truecm]{geometry}
\usepackage{amssymb, mathrsfs, tikz, graphicx, bbm, cite, amsmath, mathtools, amsthm }

\newtheorem{lem}{Lemma}[section]

\newtheorem{thm}[lem]{Theorem}

\theoremstyle{remark}
\newtheorem{rem}{Remark}
\newcommand*{\myDots}{\ifmmode\mathellipsis\else.\kern-0.07em.\kern-0.07em.\fi}

\DeclarePairedDelimiter\floor{\lfloor}{\rfloor}

\allowdisplaybreaks
\begin{document}
\title[Sums of product of  consecutive partial quotients]
{Limit theorems for sums of products of consecutive partial quotients of continued fractions}
\author{Hui Hu}
\address{School of Mathematics and Information Science, Nanchang Hangkong University, Nanchang, Jiangxi 330063, PR China\\
and
Department of Mathematics and Statistics,
La Trobe University,
Bendigo 3552,
Australia}
\email{hh5503@126.com}

\author{Mumtaz Hussain}
\address{Department of Mathematics and Statistics,
La Trobe University,
Bendigo 3552,
Australia}
\email{m.hussain@latrobe.edu.au}

\author{Yueli Yu}

\address{School of Mathematics and Statistics, Wuhan University, Wuhan, Hubei 430072, PR China}
\email{yuyueli@whu.edu.cn}

\begin{abstract}
Let $[a_1(x),a_2(x),\ldots, a_n(x), \ldots]$ be the continued fraction expansion of an irrational number $x\in (0, 1)$. The study of the growth rate of the product of consecutive partial quotients $a_n(x)a_{n+1}(x)$ is associated with the  improvements to Dirichlet's theorem (1842). We establish both the weak and strong laws of large numbers for the partial sums $S_n(x)= \sum_{i=1}^n a_i(x)a_{i+1}(x)$ as well as, from a multifractal analysis point of view,  investigate its increasing rate. Specifically, we prove the following results:
\medskip
\begin{itemize}

\item For any $\epsilon>0$, the Lebesgue measure of the set
$$\left\{x\in(0, 1): \left|\frac{ S_n(x)}{n\log^2 n}-\frac1{2\log2}\right|\geq \epsilon\right\}$$tends to zero as $n$ to infinity.

\item  For Lebesgue almost all $x\in (0,1)$,

$$\lim\limits_{n\rightarrow \infty} \frac{S_n(x)-\max\limits_{1\leq i \leq n}a_i(x)a_{i+1}(x)}{n\log^2n}=\frac{1}{2\log2}.$$

\item The Hausdorff dimension of the set $$E(\phi):=\left\{x\in(0,1):\lim\limits_{n\rightarrow \infty}\frac{S_n(x)}{\phi(n)}=1\right\}$$
is determined for a range of increasing functions $\phi: \mathbb N\to \mathbb R^+$.

\end{itemize}
\end{abstract}


   %--------------?
\maketitle

\addtocounter{section}{0}

\section{Introduction}
Let $T: [0,1)\to [0,1)$ be the Gauss map defined by
\[T(0)=0, ~ T(x)=\frac{1}{x}-\floor*{\frac{1}{x}} \textmd{ for }x\in(0,1),\]
where $\lfloor \xi\rfloor$ denotes the integer part of $\xi$. Each irrational number $x\in(0,1)$ has a unique simple continued fraction expansion
as follows:
\begin{equation*}
x=\frac{1}{a_{1}(x)+\displaystyle{\frac{1}{a_{2}(x)+\displaystyle{\frac{1}{
a_{3}(x)+_{\ddots }}}}}}:=[a_{1}(x),a_{2}(x),a_{3}(x),\ldots ],
\end{equation*}
where $a_1(x)=\lfloor 1/x \rfloor$, $a_{n}(x)= a_1(T^{n-1}(x))$ for $n\ge2$ and the positive integer $a_n(x)$ is called the $n$th partial quotient of $x$.
There are various metrical results regarding the behaviour of the sum of partial quotients, $\mathcal S_n(x):=\sum_{i=1}^na_i(x)$,  of continued fractions. What may be classed as the first significant result  is attributed to  Khinchin \cite{Khinchin1935} who proved the following weak law of large numbers with the normalising function $n\mapsto~n\log n$. Throughout the paper, we will use $\lambda(A)$ to denote the Lebesgue measure of a set $A$.
\begin{thm}[Khinchin, 1935]\label{Kh1935}  For any $\epsilon>0$,
\begin{equation*}\label{Khi}
\lambda\left\{x\in(0, 1): \left|\frac{\mathcal S_n(x)}{n\log n}-\frac1{\log2}\right|\geq \epsilon\right\}\ \parbox{.75cm}{\rightarrowfill}   \ 0 \ \text{as} \ n \ \parbox{.75cm}{\rightarrowfill} \  \infty.
\end{equation*}
\end{thm}
Khinchin's theorem illustrates that $\mathcal S_n(x)/(n\log n)$ converges, in Lebesgue measure $\lambda$, to $\frac{1}{\log2}$. In this paper, the measure implied by the statement `almost everywhere' will always be the Lebesgue measure.

For the strong law of large numbers, Philipp \cite{Philipp88} proved that there is no reasonably regular function $\phi:\mathbb N\to\mathbb R_+$ such that
$\mathcal S_n(x)/\phi(n)$ almost everywhere converges to a finite nonzero constant. However,  if the largest partial quotient $a_k(x)$ is removed from $\sum_{i=1}^na_n(x)$ then Diamond \& Vaaler \cite{DiamondVaaler} showed that the following strong law of large numbers holds with the normalising function $n\mapsto n\log n$.
\begin{thm}[Diamond--Vaaler, 1986]\label{DV}
%For any $\epsilon>0$
%\begin{equation}\label{DiamondVaaler}
%\lambda\left\{x\in(0, 1): \left|\frac{\sum_{i=1}^na_i(x)-\max\limits_{1\le i\le n}a_i(x)}{n\log n}-\frac1{\log2}\right|\geq \epsilon\right\}\to 0.
%\end{equation}
  For almost every $x\in (0,1)$,
  \[\lim\limits_{n\to\infty}\frac{\mathcal S_n(x)-\max\limits_{1\le i\le n}a_i(x)}{n\log n}=\frac{1}{\log 2}.\]
\end{thm}
%See \cite{Philipp} for more limit theorems about the sum of partial quotients.
To further analyse the behaviour of the sum $\mathcal S_n(x)$, in particular its increasing rate,  a focus has been on the Hausdorff dimension, $\dim_\mathcal H$,  of related level sets
\[B(\phi)=\left\{x\in(0,1): \lim\limits_{n\to\infty}\frac{\mathcal S_n(x)}{\phi(n)}=1\right\}.\]
Here and throughout the paper, the function $\phi: \mathbb{N}\to \mathbb{R}_+$ is monotonically increasing and $\lim_{n\to\infty}\phi(n)=\infty$.
There are several results regarding the Hausdorff dimension of $B(\phi)$ for different functions $\phi$. It was proved by Xu \cite{Xu08} that
\begin{equation*}
\dim_{\mathcal H} B(\phi)=
\left\{
             \begin{array}{lll}
             \frac{1}{2} & {\rm if} & \phi(n)=e^n,\\ [2ex]
             \frac{1}{b+1} & {\rm if} & \phi(n)=e^{b^n},b>1.
             \end{array}
\right.
\end{equation*}
The proof in \cite{Xu08} also implies that $\dim_{\mathcal H} B(\phi)=1/2$ if $\phi(n)=e^{n^{\gamma}}$ for any $\gamma\ge1$. Wu and Xu \cite{WuXu} proved that $\dim_{\mathcal H} B(\phi)=1$ for some general function $\phi$. In particular, if $\phi(n)=e^{n^{\gamma}}$ with $0<\gamma<1/2$, then $\dim_{\mathcal H} B(\phi)=1$. While the case $\phi(n)=e^{n^{\gamma}}$ with $1/2\le \gamma<1$ was recently solved by Liao and Rams \cite{LiaoRams}. They proved that $\dim_{\mathcal H} B(\phi)=1/2$ in this case. Thus for $\phi(n)=e^{n^{\gamma}}$ and $0<\gamma<\infty$, the dimension function $\dim_{\mathcal H} B(\phi)$ is discontinuous at $\gamma=1/2$. Finally, Iommi and Jordan \cite{Iommi-Jordan} investigated the case $\phi(n)=cn$ with $c\ge1$. We refer the reader to  \cite[pp. 403]{LiaoRams} for a graph of Hausdorff dimension of $B(\phi)$ for different values of $\phi$. The graph illustrates an interesting phenomena that the functional $\phi\mapsto \dim_\mathcal H B(\phi)$ is increasing for small values of $\phi$ and decreasing for large values of $\phi$.

In this paper, we are interested in the analogues of the results stated above by replacing the sum of partial quotients with the sum of products of consecutive partial quotients.  This consideration is motivated by the recent developments in the theory of uniform Diophantine approximation, specifically to the set of real numbers admitting improvements to Dirichlet's theorem. Let $\varphi$ be a monotonically non-increasing function. The set $\mathcal D(\varphi)$ of $\varphi$-Dirichlet improvable numbers  is the set of all $x \in \mathbb{R}$ such that
\[|qx-p|< \varphi(t), ~ 1\le |q|< t\]
has a nonzero integer solution $(p, q)$ for all large enough $t$. The set  $\mathcal D(\varphi)$ has an elegant characterisation in terms of growth of product of consecutive partial quotients as
   \begin{align*}&\left\{x\in(0,1): a_n(x)a_{n+1}(x)>\widetilde{\varphi}(q_n(x))\textmd{ for i.m. }n\in\mathbb N\right\} \subset \mathcal{D}^c(\varphi) \\ &\subset \left\{x\in(0,1): a_n(x)a_{n+1}(x)\geq \widetilde{\varphi}(q_n(x))/4 \textmd{ for i.m. }n\in\mathbb N\right\}
   \end{align*}
where $\widetilde{\varphi}(r)=\frac{r\varphi(r)}{1-r\varphi(r)},$ $p_n(x)/q_n(x)=[a_1(x),a_2(x),\ldots,a_n(x)]$ is the $n$-th convergent of the continued fraction expansion of $x$,  and `i.m.' stands for infinitely many.  The Lebesgue measure of $\mathcal{D}^c(\varphi)$ has been determined in \cite{KleinbockWadleigh} and the Hausdorff measure and dimension have been obtained in \cite{HKWW, BHS}. The study of comparisons of the set of Dirichlet non-improvable numbers with that of the set of well-approximable numbers was carried out in \cite{BBH1, BBH2}, and the study of level sets about the growth rate of $\{a_n(x)a_{n+1}(x): n\geq 1\}$ relative to that of $\{q_n(x): n\geq 1\}$ was discussed in \cite{HuWu}.  In particular, to get the Lebesgue measure of $\mathcal D^c(\varphi)$, Kleinbock and Wadleigh \cite{KleinbockWadleigh} obtained the Lebesgue measure of the set
\begin{equation*}
G(\Phi):=\left\{x\in(0,1): a_n(x)a_{n+1}(x)>\Phi(n)\textmd{ for i.m.  }n \in \mathbb N\right\}
\end{equation*}
and then  as a corollary, they deduced the Lebesgue measure of $\mathcal D^c(\varphi)$. See also \cite{HuWuXu} for a detailed analysis of a generalised version of the set $G(\Phi)$.
\begin{thm}[{\cite[Theorem 3.6]{KleinbockWadleigh}}]\label{K-W-Borel-Bernstain}
  Let $\Phi:\mathbb{N}\to[1,\infty)$ be a function with $\lim_{n\to\infty}\Phi(n)=\infty$.
   Then the Lebesgue measure of $G(\Phi)$
   is  either zero or full according as the series ${\sum_{n=1}^{\infty }}\log\Phi(n)/\Phi (n)$ converges or diverges respectively.
\end{thm}

It is thus clear that the study of growth of product of consecutive partial quotients is pivotal in providing information on the set of Dirichlet non-improvable numbers. In this paper, we study the sum of product of consecutive partial quotients. Throughout this paper, we specify
\[S_n(x):= \sum_{i=1}^n a_i(x)a_{i+1}(x)\]
for $n\ge1$ and irrational $x\in(0,1)$.  Our first result  (akin to Khinchin's Theorem \ref{Kh1935}) shows that $S_n(x)$ converges, in Lebesgue measure,  to $1/(2\log2)$ with the normalising function $n\mapsto n\log^2 n$.  Hence we have the following weak law of large numbers for the sum $S_n(x)$.

 \begin{thm}\label{convergent in measure}For any $\epsilon>0$,
\begin{equation*}\label{Khi}
\lambda\left\{x\in(0, 1): \left|\frac{ S_n(x)}{n\log^2 n}-\frac1{2\log2}\right|\geq \epsilon\right\}\ \parbox{.75cm}{\rightarrowfill}   \ 0 \ \text{as} \ n \ \parbox{.75cm}{\rightarrowfill} \  \infty.
\end{equation*}
\end{thm}

\begin{rem}
Note\footnote{We are thankful to an anonymous referee for suggesting this remark} that, in view of this theorem,  the sum $S_n(x)$ grows faster than any linear increasing speed. One of the reasons for this is that the function $x\mapsto a_n(x)a_{n+1}(x)$ is not in $L^1$. Since $S_n(x)$ is an ergodic sum of this function, the problem can be viewed as a problem in infinite ergodic theory, namely the study of the growth of an ergodic sum $S_n\Psi(x): = \sum_{k=1}^n \Psi(T^k(x))$ for some observable $\Psi:[0, 1]\to \mathbb R$ which is not in $L^1$. A result of Aaronson \cite{Aaronson} states that there is no normalisation $b_n$ such that $S_n(x)/b_n$ converges to a non-zero number for almost all $x$. However, in many cases,  it is possible to prove that for almost every  $x$,  $b_n\leq S_n\Psi(x)\leq c_n$ holds for sufficiently large $n$, where $b_n$ and $c_n$ grows almost at the same speed  \cite{GHPZ}.

Within this general setting, by trimming the sum $S_n(x)$, that is by removing one or several of its largest terms, then it is possible to normalise the sum so that it converges to a non-zero number for almost every $x$.  There are for instance related work by Kesseb\"ohmer-Schindler \cite{KS1, KS2}  which proves such results in a rather general setting.
\end{rem}
In contrast, however, note that by Theorem \ref{K-W-Borel-Bernstain},  for almost every $x$ the inequality $$a_{n}(x)a_{n+1}(x)\ge n\log^2n\log\log n$$
 holds for infinitely many $n$. Thus, it demonstrates that $S_n(x)/(n\log^2n)$ does not converge to $\frac{1}{2\log 2}$ almost everywhere. Similar to  Philipp \cite{Philipp88}, it can be straightforwardly proved that there is no reasonably regular function $\phi$ such that
$S_n(x)/\phi(n)$ almost everywhere converges to a finite nonzero constant. However, if we remove the largest term from $S_n(x)$ then an analogue of Diamond-Vaaler Theorem \ref{DV} for the sum $S_n(x)$ holds.
\begin{thm}\label{main thm-measure}
For almost every $x\in (0,1)$, we have
$$\lim\limits_{n\rightarrow \infty}\frac{S_n(x)-\max\limits_{1\leq i\leq n}a_i(x)a_{i+1}(x)}{n\log^2 n}=\frac{1}{2\log 2}.$$
\end{thm}
%
%There have been many results on the Hausdorff dimensions of some sets defined by the growth speed of $a_n(x)a_{n+1}(x)$ or defined by the relative growth speed of $a_n(x)a_{n+1}(x)$ and other variables. In \cite{HuWuXu}, Huang, Wu and Xu completely determined
%the Hausdorff dimensions of the set
%\[\{x\in [0,1):  a_n(x)a_{n+1}(x)\ge \phi(n)\textmd{ for infinitely many } n \}\]
%and a generalised set. See \cite{BBH1,BBH2,HuWU,HKWW} for more dimensional results related to the growth speed of $a_n(x)a_{n+1}(x)$.
We further analyse the fractal structure of $S_n(x)$ with respect to an increasing function $\phi$ by considering the set
\[E(\phi):=\left\{x\in(0,1):\lim\limits_{n\to\infty}\frac{S_n(x)}{\phi(n)}=1\right\}.\]
%where $\phi:\mathbb{N}\to \mathbb{R}_+$ is an increasing function with $\lim_{n\to\infty}\phi(n)/n=\infty$.
We get the following three results about the Hausdorff dimension of $E(\phi)$.

\begin{thm}\label{thm-dimension-1}
  Let $\phi:\mathbb{N}\to\mathbb{R}_+$ be a monotonically increasing positive function with
  \[\lim\limits_{n\to\infty}\frac{\phi(n)}{n}=\infty, \lim\limits_{n\to\infty}\frac{\phi(n+1)}{\phi(n)}=1, \limsup\limits_{n\to\infty}\frac{\log \log \phi(n)}{\log n}<1/2.\]
  Then we have $\dim_{\mathcal H} E(\phi)=1$.
\end{thm}

It is probably worth emphasising that the conditions on $\phi$ are more general than simply stating that $\phi(n)=e^{n^\alpha}$ for $\alpha\in (0, 1/2)$.
\begin{thm}\label{thm-dimension-1/1+alpha}
  For any $\alpha>1$,  if $\phi(n)=e^{{\alpha}^n}$, then $\dim _\mathcal H E(\phi)=\frac{1}{1+\alpha}$.
\end{thm}

%In this theorem $\alpha$ has to be strictly greater than $1$ since $\phi$ is monotonically increasing.

\begin{thm}\label{thm-dimension-1/2}
  For any $\alpha\ge1/2$,  if $\phi(n)=e^{n^{\alpha}}$, then $\dim_{\mathcal H} E(\phi)=1/2$.
\end{thm}
%
%To summarise these results, we illustrate the Hausdorff dimension graph (figure 1) for different functions $\phi$.
%

\medskip

We illustrate a summary of these Hausdorff dimension results for different $\phi$ in the figure below.

\medskip

\medskip

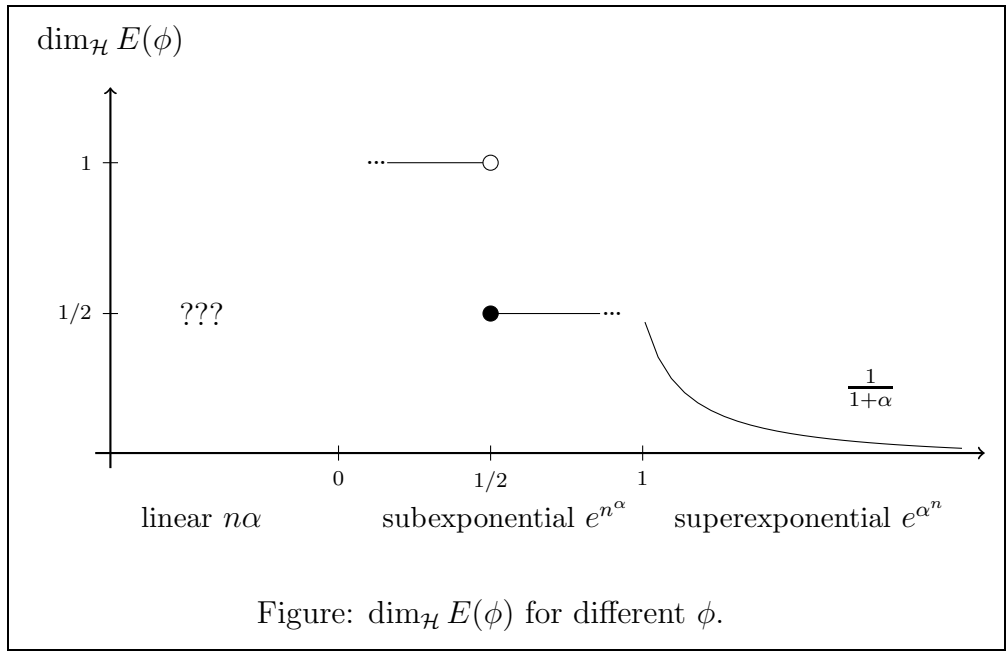
\begin{center}
\fbox{
\begin{tikzpicture}[scale=2]
\draw[->, thick] (0.4, 0.075)--(6.25, 0.075);
\draw[->, thick] (0.5, -0.025)--(0.5, 2.5)node[above, yshift=0.25cm]{$\dim_\mathcal H E(\phi)$};
\draw[-] (0.45, 1)node[left]{\tiny$1/2$}--(0.55, 1);
\draw[-] (0.45, 2)node[left]{\tiny$1$}--(0.55, 2);

\draw[-] (2, .025)node[below]{\tiny$0$}--(2, 0.125);

\draw[-] (3, 0.025)node[below]{\tiny$1/2$}--(3, 0.125);

\draw[-] (4, 0.025)node[below]{\tiny$1$}--(4, 0.125);

\draw [fill=black] (3, 1) circle (0.05);
\draw [-](3,1)--(3.72, 1);
%\draw[dashed] (3.6, 1) --(4,1);

\draw [-] (3, 2) circle (0.05);

\draw [-](2.95,2)--(2.32, 2);
%\draw[dashed] (2.4, 2) --(2,2);

\node at (2.25, 2) {\myDots};
\node at (3.8, 1) {\myDots};

 \draw[xshift=4cm, domain=0.015:2.1] plot (\x,{1/(4*\x+1)});
\node at (5.5, .5){$\frac1{1+\alpha}$};

\node at (5.1, -.35){\text{\small superexponential} $e^{\alpha^n}$};

\node at (3.1, -.35){\text{\small subexponential} $e^{n^\alpha}$};

\node at (1.1, -.35){\text{\small linear} $n\alpha$};
\node at (1.1, 1){$???$};

\node at (3, -1) {\text{Figure: $\dim_\mathcal H E(\phi)$ for different $\phi$.}};
\end{tikzpicture}
}

\end{center}

\medskip
\medskip

We prove these results for a slightly more general setting  (see Lemmas \ref{lower-bound-1/2}, \ref{lower-bound-1/1+alpha} and \ref{upper-bound-1/2-0} below).
%The  proofs of the above three results, we give upper and lower bounds of $\dim_{\mathcal H} E(\phi)$ for a little more general cases (See  Lemmas \ref{lower-bound-1/2}, \ref{lower-bound-1/1+alpha} and \ref{upper-bound-1/2-0} below).
\begin{rem}Naturally one may wonder what happens when $\phi$ is a linear function such as $\phi(n)=cn$ for some constant $c$.
That is the multifractal analysis of the Birkhoff average $\frac{1}{n}\sum_{i=0}^{n-1} f(T^i(x))$ with the potential function $f(x)=a_1(x)a_2 (x)$. For an expanding map with infinite branches such as the Gauss map $T$, the known results, such as the multifractal analysis of $\mathcal S_n(x)$,  are all based on some regularity conditions of the potential function $f$. For further  details see the summable variation condition in \cite{Iommi-Jordan} and the variation uniformly converging to $0$ in the paper \cite{FanJoranLiaoRams}. Since the above regular conditions do not hold for $f(x)=a_1(x)a_2(x)$, the known conclusions and methods are not applicable in this case. Therefore, we believe some ingenuity and new arguments are needed to resolve this case.
%It turns out that the methods presented in this paper are not applicable to functions of the form $\phi(n)=cn$  partly {\color{red} because the function $f(x)=a_1(x)a_2(x)$ does not satisfy the usual regularity conditions. would it be possible to explain it a bit more}

\end{rem}

\begin{rem}
One may wonder, if the results stated above and their methods of proof extends to the sum $\sum_{i=1}^na_i(x)a_{i+1}(x)\cdots a_{i+k}(x)$ of the products of the $k+1$ consecutive partial quotients. We believe that with obvious modifications,  the results and methods of proofs will be similar. The calculations, however, will be a bit lengthy without yielding any new information.
\end{rem}

\medskip

The paper is organised as follows. In Section \ref{prepare}, we introduce some notation and discuss some basic properties of continued fraction and Gauss measure. In Section \ref{proof for measure}, we prove Theorems \ref{convergent in measure} and \ref{main thm-measure}. In Section \ref{proof-for-full-dimension}, we prove Theorem \ref{thm-dimension-1} and in Section \ref{proof-of-last-two} we prove Theorems \ref{thm-dimension-1/1+alpha} and \ref{thm-dimension-1/2}.
\medskip

\noindent {\bf Acknowledgements.}  %We thank the anonymous referees for their detailed comments and remarks that have improved some of the proofs and presentation of the paper.  
The first-named author was supported by the Natural Science Foundation of China (11701261) and the China Scholarship Council. The second-named author was supported by the Australian Research Council Discovery Project (200100994).

\section{Preliminaries}\label{prepare}
We list some basic properties of continued fractions. For more details of continued fractions, we refer to \cite{Khinchin_book64}. The definitions and properties of Hausdorff measure and dimension can be found in \cite{Falconer_book2003}.

For any $n\ge1$ and any positive integers $a_1,\cdots,a_n$, we call
\[I(a_1,\ldots,a_n):=\{x\in[0,1]: a_1(x)=a_1,\ldots,a_n(x)=a_n\}\]
an $n$-th order {\em cylinder}.
The cylinder $I(a_1,\ldots,a_n)$ is an interval and its length satisfies
 \begin{equation}\label{length of cylinders}
 |I(a_1,\ldots,a_n)|=q_n^{-1}(q_n+q_{n-1})^{-1}
 \end{equation}
where  $q_i, 1\le i\le n$  satisfy the recursive formula
\begin{equation}\label{recursive formulae of qn}
  q_{-1}=0, q_0=1, q_{i}=a_{i}q_{i-1}+q_{i-2}.
\end{equation}
Moreover, by \eqref{length of cylinders} and \eqref{recursive formulae of qn}, we have
\[\frac{1}{2q_n^2}\le |I(a_1,\ldots,a_n)|\le \frac{1}{q_n^2}\]
and
\begin{equation}\label{estimate of length of cylinder by partial quotients}
  2^{-(2n+1)}\prod\limits_{k=1}^n a_k^{-2}\le |I(a_1,\ldots,a_n)|\le \prod\limits_{k=1}^n a_k^{-2}.
\end{equation}
For any integer $k\ge1$, the first order cylinder $I(k)$ satisfies
\[I(k)=\left(\frac{1}{k+1},\frac{1}{k}\right].\]
For any integers $i,j\ge1$, the cylinder $I(i,j)$ satisfies
\[I(i,j)=\left[\frac{1}{i+\frac{1}{j}},\frac{1}{i+\frac{1}{j+1}}\right).\]
We will be calculating the gap between cylinders and to do that we shall use the following fact. Notice that  $I(a_1,\cdots,a_{n-1},a_n)$ and $I(a_1,\cdots,a_{n-1},a_n+1)$ are adjacent subintervals of $\mathbb{R}$.  If $n$ is odd (respectively even), then $I(a_1,\cdots,a_{n-1},a_n)$ is on the right (respectively left) side of $I(a_1,\cdots,a_{n-1},a_n+1)$ in $\mathbb{R}$, see \cite{Khinchin_book64}.
%All these facts of continued fraction are contained in  \cite{Khinchin_book64}.
%%If $n$ is even, the former is on the right side of the latter.

\medskip

Let $\mu$ be the Gauss measure on [0,1] defined by
\[\mu(A)=\frac{1}{\log2 }\int_A\frac{1}{1+x}dx\]
for any Lebesgue measurable set $A$. The Gauss measure $\mu$ is $T$-invariant and equivalent to the Lebesgue measure $\lambda$.
The following exponentially mixing property of Gauss measure is well known (see \cite{Billingsley} or \cite{Philipp88}).
\begin{lem}\label{mixing-estimate}
  There exists a constant $0<\rho<1$ such that
  $$\mu \left(I(a_1,a_2,\ldots,a_m)\cap T^{-m-n}B)\right)=\mu\left(I(a_1,a_2,\ldots,a_m)\right)\mu(B)(1+O({\rho}^n))$$
  for any $m\ge1, n\ge0$, any $m$-th cylinder $I(a_1,a_2,\ldots,a_m)$ and any Borel set $B$, where the implied constant in $O({\rho}^n)$ is absolute.
\end{lem}

%
% Recall that $\varphi$ is always a positive function defined on $\mathbb{N}$ with $\lim\limits_{n\rightarrow \infty}\varphi(n)/n=\infty$. We denote by  $\lceil\xi\rceil$ the smallest integer no less than $\xi\in \mathbb{R}$ and $\lfloor \xi \rfloor$ stands for the integral part. Here $\mu$ and $\lambda$ are
%Gauss measure and Lebesgue measure respectively.

\section{Proofs of Theorems  \ref{convergent in measure} and  \ref{main thm-measure} }\label{proof for measure}
We adopt  strategies from \cite{DiamondVaaler} and \cite{Philipp88} to prove Theorems  \ref{convergent in measure} and  \ref{main thm-measure}.   The proof of the following lemma is similar to that of \cite[Theorem 3.6]{KleinbockWadleigh}. Let $\varphi: \mathbb{N}\to [1,\infty)$ be a positive function.

\begin{lem}\label{measure of a_1a_2}
Let $A_n:=\{x\in(0,1): a_1(x)a_2(x)\geq \varphi(n)\}$ with $n\geq 1$. Then  the Gauss measure of $A_n$ satisfies
$$  \mu(A_n)= \frac{1}{\log2}\cdot \frac{\log\varphi(n)+O(1)}{\varphi(n)}.$$
%when $n$ is large enough.
\end{lem}
\begin{proof}
For any $n\geq 1$,
$$
\begin{aligned}
A_n=&\bigcup\limits_{1\le a\leq \varphi(n)}\left[\cfrac{1}{a+1/\lceil \cfrac{\varphi(n)}{a}\rceil},\frac{1}{a}\right)\bigcup
\left(\bigcup\limits_{a> \varphi(n)}\left[\frac{1}{a+1},\frac{1}{a}\right)\right)  \\
\subseteq & \bigcup\limits_{a\leq \varphi(n)}\left[\frac{1}{a+\frac{a}{\varphi(n)}},\frac{1}{a}\right)\bigcup \left(0,\frac{1}{\varphi(n)}\right).
\end{aligned}
$$
Since
$$
\begin{aligned}
\mu\left(\left[\frac{1}{a+\frac{a}{\varphi(n)}},\frac{1}{a}\right)\right)&=\frac{1}{\log2}\cdot\int_{\frac{1}{a+\frac{a}{\varphi(n)}}}^{\frac{1}{a}}\frac{1}{1+x}dx
\\ &=\frac{1}{\log2}\cdot \log\left(1+\frac{1}{a(\varphi(n)+1)+\varphi(n)}\right)  \\
&\leq \frac{1}{\log2}\cdot \frac{1}{a(\varphi(n)+1)+\varphi(n)}
\end{aligned}
$$
 and
 $$\begin{aligned}
   \mu\left(\left(0,\frac{1}{\varphi(n)}\right)\right)=\frac{1}{\log 2}\int_0^{1/\varphi(n)}\frac{dx}{1+x}=\frac{\log(1+1/\varphi(n))}{\log2}\le \frac{1}{\varphi(n)\log 2},
 \end{aligned}$$
we have
$$
\begin{aligned}
\mu(A_n) &\leq \frac{1}{\log2}\cdot \sum\limits_{a=1}^{\lfloor \varphi(n) \rfloor} \frac{1}{a(\varphi(n)+1)+\varphi(n)}+\frac{1}{\varphi(n)\log 2} \\
&\leq \frac{1}{\log2}\cdot \left(\frac{1}{2\varphi(n)+1}+\int_{1}^{\varphi(n)}  \frac{dx}{(\varphi(n)+1)x+\varphi(n)}\right)+\frac{1}{\varphi(n)\log 2} \\
& \leq \frac{1}{\log2}\cdot \frac{\log\varphi(n)+O(1)}{\varphi(n)}.
\end{aligned}
$$
On the other hand,
$$A_n \supset \bigcup\limits_{a\leq \varphi(n)}\left(\frac{1}{a+\frac{a}{a+\varphi(n)}},\frac{1}{a}\right).$$
Then we have
$$
\begin{aligned}
\mu(A_n) &\geq \frac{1}{\log2}\cdot \sum\limits_{a=1}^{\lfloor \varphi(n) \rfloor} \left(\log\left(1+\frac{1}{a}\right)-\log\left(1+\frac{1}{a+\frac{a}{a+\varphi(n)}}\right)\right)  \\
&= \frac{1}{\log2}\cdot \sum\limits_{a=1}^{\lfloor \varphi(n) \rfloor}\log\left(1+\frac{\frac{1}{a+\varphi(n)+1}}{a+\frac{a+\varphi(n)}{a+\varphi(n)+1}}\right)   \\
& \geq \frac{1}{\log2}\cdot \sum\limits_{a=1}^{\lfloor \varphi(n) \rfloor}\log\left(1+\frac{1}{(a+1)(a+\varphi(n)+1)}\right).
\end{aligned}
$$
Thus
$$
\begin{aligned}
\mu(A_n) & \geq \frac{1}{\log2}\cdot \sum\limits_{a=1}^{\lfloor \varphi(n) \rfloor} \left(\frac{1}{(a+1)(a+\varphi(n)+1)}- \frac{1}{2(a+1)^2(a+\varphi(n)+1)^2}\right) \\
&\ge\frac{1}{\log2}\cdot \sum\limits_{a=1}^{\lfloor \varphi(n) \rfloor} \left(\frac{1}{(a+1)(a+\varphi(n)+1)}- \frac{1}{2\varphi(n)^2}\right)\\
&\geq \frac{1}{\log2} \left(\int_{1}^{\lfloor \varphi(n) \rfloor}\frac{dx}{(x+1)(x+\varphi(n)+1)}\right)-\frac{1}{2\varphi(n)}  \\
&\geq \frac{1}{\log2} \frac{\log \varphi(n)+O(1)}{\varphi(n)},
\end{aligned}
$$
which completes the proof.
\end{proof}

To simplify notation,  for any irrational $x\in (0,1)$ and $i\geq 1$,  let
\[b_i(x)=a_i(x)a_{i+1}(x).\]
 % By  similar arguments as in \cite{DiamondVaaler}, we have the following.
  For any real quantities $\xi$ and $\eta$, we use the notation $\xi\ll \eta$ if there is an unspecified constant $c$ such that $\xi\leq  c\eta$. The following lemma is an analogue of Lemma 2 in \cite{DiamondVaaler}.
\begin{lem}\label{only one large term}
Let $c>3/2$ and $g(n)=n\log^c n$. Then for Lebesgue almost all $x\in (0,1)$, there exists a positive integer $n_0(x)$ such that
$$\#\{1\leq i \leq n: b_i(x)>g(n)\}\leq 1$$
for all $n\geq n_0(x)$.
\end{lem}

\begin{proof}
For any $k\geq 1$, let
$$B_k=\bigcup\limits_{1\leq i< j\leq 2^{k+1}}\{x\in(0,1): b_i(x)>g(2^k), b_j(x)>g(2^k)\}$$ and
$$B_k(i,j)=\{x\in(0,1): b_i(x)>g(2^k), b_j(x)>g(2^k)\}.$$ Then
$$B_k=\bigcup_{\substack{1\leq i< j\leq 2^{k+1} \\ j-i\geq 2 }}B_k(i,j) \bigcup \bigcup_{\substack{1\leq i< j\leq 2^{k+1} \\ j-i=1 }}B_k(i,j)$$
and
$$\mu(B_k)\le \sum_{\substack{1\leq i< j\leq 2^{k+1} \\ j-i\geq 2 }}\mu(B_k(i,j))+\sum_{\substack{1\leq i\leq 2^{k+1}-1}}\mu(B_k(i,i+1))$$
By Lemmas \ref{mixing-estimate} and \ref{measure of a_1a_2},  %there exists $0<\rho <1$ such that
%\begin{equation}\label{measure-1}
\begin{align}\label{measure-1}
&\sum_{\substack{1\leq i< j\leq 2^{k+1} \\ j-i\geq 2 }}\mu(B_k(i,j)) \notag\\
= &\sum_{\substack{1\leq i< j\leq 2^{k+1} \\ j-i\geq 2 }}\mu(\{x\in(0,1): b_i(x)>g(2^k)\})\cdot
\mu(\{x\in(0,1): b_j(x)>g(2^k)\})\cdot (1+O(\rho^{j-i-2})) \notag \\
\ll &(2^{k+1})^2\cdot \frac{(\log g(2^k))^2}{(g(2^k))^2}
+\sum_{\substack{1\leq i< j\leq 2^{k+1} \\ j-i\geq 2 }}\frac{(\log g(2^k))^2}{(g(2^k))^2}\cdot \rho^{j-i-2} \notag\\
\ll &(2^{2(k+1)}+2^{k+1})\cdot \frac{(\log g(2^k))^2}{(g(2^k))^2}.
\end{align}
%\end{equation}

Note that
$$\begin{aligned}
\mu(B_k(i,i+1))=&\mu\left(\{x\in(0,1): b_i(x)>g(2^k), b_{i+1}(x)>g(2^k)\}\right) \\
=&\mu\left(\{x\in(0,1): a_1(x)a_2(x)>g(2^k), a_2(x)a_3(x)>g(2^k)\}\right) \\
=&\mu\left(\left\{x\in(0,1): a_2(x)\leq g(2^k), a_1(x)> \frac{g(2^k)}{a_2(x)}, a_3(x)> \frac{g(2^k)}{a_2(x)}\right\}\right)\\
 &\qquad +\mu\left(\{x\in(0,1): a_2(x)>g(2^k)\}\right).
\end{aligned}$$
We have
$$\begin{aligned}
  \mu(\{x\in(0,1): a_2(x)>g(2^k)\}&=\mu(\{x\in(0,1): a_1(x)>g(2^k)\}\\
&\ll \lambda(\{x\in(0,1): a_1(x)>g(2^k)\} \\ &\ll \frac{1}{g(2^k)},
\end{aligned}$$
and
$$\begin{aligned}
&\mu\left(\{x\in(0,1): a_2(x)\leq g(2^k), a_1(x)> \frac{g(2^k)}{a_2(x)}, a_3(x)> \frac{g(2^k)}{a_2(x)}\}\right) \\
=&\sum\limits_{1\leq m\leq g(2^k)} \sum\limits_{i>\frac{1}{m} g(2^k)}  \sum\limits_{j>\frac{1}{m}g(2^k)}
\mu\left(\{x\in(0,1): a_2(x)=m, a_1(x)=i, a_3(x)=j\}\right) \\
\ll & \sum\limits_{1\leq m\leq g(2^k)} \sum\limits_{i>\frac{1}{m}g(2^k)} \sum\limits_{j>\frac{1}{m}g(2^k)}
\lambda\left(\{x\in(0,1): a_2(x)=m, a_1(x)=i, a_3(x)=j\}\right) \\
\ll & \sum\limits_{1\leq m\leq g(2^k)} \sum\limits_{i>\frac{1}{m}g(2^k)} \sum\limits_{j>\frac{1}{m}g(2^k)}
\frac{1}{m^2}\cdot \frac{1}{j^2} \cdot \frac{1}{i^2} \\
\ll & \sum\limits_{1\leq m\leq g(2^k)} \frac{m}{g(2^k)}\cdot \frac{m}{g(2^k)} \cdot \frac{1}{m^2}\\ &\ll \frac{1}{g(2^k)}.
\end{aligned}$$
Thus
$$\mu(B_k(i,i+1))=\mu(\{x\in(0,1): b_i(x)>g(2^k), b_{i+1}(x)>g(2^k)\})\ll \frac{1}{g(2^k)},$$
which implies that
\begin{equation}\label{measure-2}
\sum\limits_{1\leq i \leq 2^{k+1}-1}\mu(B_k(i,i+1))\ll \frac{2^{k+1}}{g(2^k)}.
\end{equation}
By inequalities \eqref{measure-1} and \eqref{measure-2}, we have
$$\begin{aligned}
\mu(B_k)&\ll (2^{2(k+1)}+2^{k+1})\cdot \frac{(\log g(2^k))^2}{(g(2^k))^2}+\frac{2^{k+1}}{g(2^k)} \\
&\ll k^{2-2c}+k^{-c},
\end{aligned}$$
which implies that $\sum_{k=1}^{\infty}\mu(B_k)<\infty$ since $c>\frac{3}{2}$. By the Borel--Cantelli Lemma,  for almost every $x\in (0,1)$,
there is an integer $k_0(x)$ such that for all $k\geq k_0(x)$, $$\# \{1\leq i\leq 2^{k+1}: b_i(x)>g(2^k)\}\leq 1.$$
It follows that for any $n\geq 2^{k_0(x)}$, we have
$$\# \{1\leq i\leq n: b_i(x)>g(n)\}\leq 1.$$
\end{proof}

\begin{lem}\label{expectation}
Let $\epsilon>0$. For any $N\ge 1$, we define
$$b_i^*(x)=b^{*}_{i,N}(x)=
\left\{
              \begin{array}{ll}
            b_i(x), & {\rm if} \ \  b_i(x)\leq N(\log N)^{\frac{3}{2}+\epsilon},\\ [2ex]
            0,      &  {\rm otherwise},\\
             \end{array}
\right.
$$
for $1\le i\le N$ and let $S_N^*(x)=\sum_{i=1}^{N}b_i^*(x)$. Then we have
$$\lim\limits_{N\rightarrow \infty}\frac{\mathbb{E}(S_N^*(x))}{N\log^2N}=\frac{1}{2\log 2},$$
where $\mathbb E$ denotes the mathematical expectation, that is,  $\mathbb E(h(x))=\int_0^1h(x)d\mu(x)$.
\end{lem}

\begin{proof}
We denote $\psi(N)=\lfloor N(\log N)^{\frac{3}{2}+\epsilon} \rfloor$. Then,
$$\begin{aligned}
\mathbb E(b_i^*(x))&=\int_{0}^{1}b_i^*(x)d\mu(x)=\sum\limits_{k=1}^{\psi(N)}k\cdot \mu\{x\in(0,1):b_i(x)=k\} \\
&=\sum\limits_{k=1}^{\psi(N)}k\cdot \left(\mu\{x\in(0,1):b_i(x)\geq k\}-\mu\{x\in(0,1):b_i(x)\geq k+1\}\right) \\
&=\sum\limits_{k=1}^{\psi(N)}\mu\{x\in(0,1):b_i(x)\geq k\}-\psi(N)\cdot \mu\{x\in(0,1):b_i(x)\geq \psi(N)+1\}.
\end{aligned}$$
By Lemma \ref{measure of a_1a_2}, we have
\begin{equation}\label{equation-1}
\mathbb E(b_i^*(x)) \leq \frac{1}{\log 2}\sum\limits_{k=1}^{\psi(N)}\frac{\log k+O(1)}{k}
\leq \frac{1}{\log 2}\cdot\left(\frac{1}{2}\log^2 \psi(N)+O(\log \psi(N))\right)
\end{equation}
and
$$\begin{aligned}
\mathbb E(b_i^*(x))
&\geq \frac{1}{\log 2}\sum\limits_{k=1}^{\psi(N)}\frac{\log k+O(1)}{k}
    -\frac{1}{\log2}\cdot \psi(N)\frac{\log (\psi(N)+1)+O(1)}{\psi(N)+1}  \\
&\geq \frac{1}{\log 2}\cdot(\frac{1}{2}\log^2 \psi(N)+O(\log \psi(N)))-\frac{1}{\log 2}\cdot(\log (\psi(N)+1)+O(1)).
   \end{aligned}$$
   Thus
\begin{equation}\label{equation-2}
\mathbb E(b_i^*(x))\ge \frac{1+o(1)}{2\log 2}\log^2 \psi(N).
\end{equation}
By \eqref{equation-1}, \eqref{equation-2} and $\mathbb E(S_N^*(x))=N\mathbb E(b_1^*(x))$, the conclusion follows.
\end{proof}

\subsection{Proof of Theorem \ref{main thm-measure}}

%\begin{proof}[Proof of Theorem \ref{main thm-measure}]
Write $$S_N^*(x)=\sum_{i=1}^{N}b_i^*(x), \qquad J_N=\mathbb E(S_N^*(x)),$$
where $b_i^*(x)$ is defined as in Lemma \ref{expectation}
 and let
\[\varphi(N)=N(\log N)^{\frac{3}{2}+\epsilon}\]
 with $0<\epsilon<\frac{1}{2}$  small enough (say $\epsilon=\frac{1}{8}$). We shall estimate $\textmd{Var}(S_N^*(x))$, the variance of $S_N^*(x)$. We first estimate the second moment of $S_N^*(x)$. We have
$$\mathbb E((S_N^*(x))^2)=\int_0^1(S_N^*(x))^2d\mu(x)=\sum\limits_{1\leq m,n\leq N}\int_0^1b_m^*(x)b_n^*(x)d\mu(x).$$
 To simplify notation we denote
\[b_{m,n}=\int_0^1b_m^*(x)b_n^*(x)d\mu(x).\]
There are three cases. %We distinguish three cases.

\medskip

\noindent{\bf Case I. }  Let $|m-n|\geq 2$. For notational convenience, let $\Lambda_{u, v}=\{x:a_u(x)a_{u+1}(x)=v\} $, then
$$\begin{aligned}
b_{m,n}&=\sum\limits_{1\leq i,j\leq \varphi(N)}ij\cdot\mu\left(\{x: a_m(x)a_{m+1}(x)=i, a_n(x)a_{n+1}(x)=j\}\right) \\
&=\sum\limits_{1\leq i,j\leq \varphi(N)}ij\cdot\mu\left(\Lambda_{m, i} \right)
\mu(\Lambda_{n, j})(1+O(\rho^{|m-n|-2}))  \\
&=\left(\int_0^1b_1^*(x)d\mu(x)\right)^2(1+O(\rho^{|m-n|-2}))  \\
&=\frac{J_N^2}{N^2}\left(1+O(\rho^{|m-n|-2})\right),
\end{aligned}$$
where $0<\rho<1$ is defined in Lemma \ref{mixing-estimate}. Thus
$$\sum\limits_{m,n\in \Lambda}b_{m,n}\le J_N^2+\frac{J_N^2}{N^2}\sum\limits_{m,n\in \Lambda}\rho^{|m-n|-2}
\ll J_N^2+\frac{J_N^2}{N},$$
where $\Lambda=\{(m,n):|m-n|\geq 2, 1\leq m,n\leq N\}$.

\noindent{\bf Case II. }  If $m=n$, we have
$$\begin{aligned}
b_{m,n}&=\int_0^1(b_1^*(x))^2d\mu(x)=\sum\limits_{1\leq k\leq \varphi(N)}k^2\cdot \mu(\{x:b_1(x)=k\})\\
&=\sum\limits_{1\leq k\leq \varphi(N)}k^2\cdot \left(\mu(\{x:b_1(x)\geq k\})-\mu(\{x:b_1(x)\geq k+1\})\right) \\
&\leq \sum\limits_{1\leq k\leq \varphi(N)}(2k-1)\mu(\{b_1(x)\geq k\})\\& \ll \sum\limits_{1\leq k\leq \varphi(N)}(2k-1)\frac{\log k+O(1)}{k}  \\
&\ll \varphi(N)\log \varphi(N).
\end{aligned}$$
So
$$\sum\limits_{1\leq m=n\leq N}b_{m,n}\ll N \varphi(N)\log \varphi(N).$$

\noindent{\bf Case III. }  If $|m-n|= 1$, we assume $n=m+1$. Then
$$\begin{aligned}
b_{m,n}&=\int_0^1b_m^*(x)b_{m+1}^*(x)d\mu(x)=\int_0^1b_1^*(x)b_{2}^*(x)d\mu(x) \\
&=\sum\limits_{s\geq 1}\sum\limits_{t\geq 1} st\cdot \mu\left(\{x\in(0,1):b_1^*(x)=s, b_2^*(x)=t\}\right)\\
&=\sum\limits_{1\leq s\leq \varphi(N)}\sum\limits_{1\leq t\leq \varphi(N)} st\cdot \mu\left(\{x\in(0,1):b_1(x)=s, b_2(x)=t\}\right) \\
&=\sum\limits_{1\leq s\leq \varphi(N)}\sum\limits_{1\leq t\leq \varphi(N)} st\cdot \mu\left(\{x\in(0,1):a_1(x)a_2(x)=s, a_2(x)a_3(x)=t\}\right) \\
&=\sum\limits_{1\leq k\leq \varphi(N)}\sum\limits_{1\leq i\leq \frac{\varphi(N)}{k}}\sum\limits_{1\leq j\leq \frac{\varphi(N)}{k}}
ik^2j\cdot \mu\left(\{x\in(0,1):a_1(x)=i, a_2(x)=k, a_3(x)=j\}\right) \\
&\ll \sum\limits_{1\leq k\leq \varphi(N)}\sum\limits_{1\leq i\leq \frac{\varphi(N)}{k}}\sum\limits_{1\leq j\leq \frac{\varphi(N)}{k}}
ik^2j\cdot \lambda(I(i,k,j)) \\
&\ll \sum\limits_{1\leq k\leq \varphi(N)}\sum\limits_{1\leq i\leq \frac{\varphi(N)}{k}}\sum\limits_{1\leq j\leq \frac{\varphi(N)}{k}}
ik^2j\cdot \frac{1}{i^2k^2j^2} \\
&\ll \sum\limits_{1\leq k\leq \varphi(N)}(1+\log \varphi (N)-\log k)^2\\ & \ll \varphi(N).
\end{aligned}$$
It follows that
$$\sum_{\substack{1\leq m,n\leq N \\ |m-n|=1 }} b_{m,n}\ll N\varphi(N).$$
Thus
$$\mathbb E((S_N^*(x))^2)=\int_0^1(S_N^*(x))^2d\mu(x)\ll J_N^2+\frac{J_N^2}{N}+N\varphi(N)\log \varphi(N)+N\varphi(N),$$
which implies that
$$\begin{aligned}
\textmd{Var}(S_N^*(x))&=\int_0^1(S_N^*(x))^2d\mu(x)-J_N^2\\
&\ll \frac{J_N^2}{N}+N\varphi(N)\log \varphi(N)\\
&\ll N\varphi(N)\log \varphi(N)\\ &
\ll N^2(\log N)^{\frac{5}{2}+\epsilon}.
\end{aligned}.$$
Take $c(k)=\lfloor e^{k^{1-\epsilon}}\rfloor$. Since
$$\begin{aligned}
&\int_0^1\sum\limits_{k=1}^{\infty}\left(S_{c(k)}^*(x)-J_{c(k)}\right)^2 (c(k))^{-2} (\log c(k))^{-4}d\mu(x) \\
= & \sum\limits_{k=1}^{\infty} \int_0^1 \left(S_{c(k)}^*(x)-J_{c(k)}\right)^2 (c(k))^{-2} (\log c(k))^{-4}d\mu(x) \\
\ll & \sum\limits_{k=1}^{\infty}\frac{(c(k))^{2}(\log c(k))^{\frac{5}{2}+\epsilon}}{(c(k))^{2} (\log c(k))^{4}}\\
\ll & \sum\limits_{k=1}^{\infty}k^{-(1-\epsilon)(\frac{3}{2}-\epsilon)}<+\infty,
\end{aligned}$$
it follows that, almost surely, we have
$$\left|S_{c(k)}^*(x)-J_{c(k)}\right|=o(1)c(k)(\log c(k))^2=o(1)J_{c(k)}.$$
That is, $$\lim\limits_{k\to\infty}\frac{S_{c(k)}^*(x)}{J_{c(k)}}=1$$
for almost every $x\in(0,1)$.
Since
$$\lim\limits_{k\to\infty}\frac{J_{c(k+1)}}{J_{c(k)}}=\lim\limits_{k\to\infty}\frac{c(k+1)(\log c(k+1))^2}{c(k)(\log c(k))^2}= 1 $$
by Lemma \ref{expectation}, we have
$$\lim\limits_{N\rightarrow \infty}\frac{S_N^*(x)}{J_N}= 1 \textmd{ for a.e. } x\in (0,1).$$
It follows that
$$\lim\limits_{N\rightarrow \infty}\frac{S_N^*(x)}{N\log^2N}=\frac{1}{2\log 2}  \textmd{ for a.e. } x\in (0,1).$$
By Lemma \ref{only one large term}, for almost every $x\in (0,1)$ and $N\ge n_0(x)$, if
$$\max\limits_{1\leq i\leq N}a_i(x)a_{i+1}(x)>N(\log N)^{\frac{3}{2}+\epsilon},$$
we have
\begin{equation}\label{main thm equa-1}
S_N^*(x)=S_N(x)-\max\limits_{1\leq i\leq N}a_i(x)a_{i+1}(x).
\end{equation}
If
$$\max\limits_{1\leq i\leq N}a_i(x)a_{i+1}(x)\leq N(\log N)^{\frac{3}{2}+\epsilon},$$
we have $S_N^*(x)=S_N(x)$ and
\begin{equation}\label{main thm equa-2}
S_N^*(x)-N(\log N)^{\frac{3}{2}+\epsilon}\leq S_N(x)-\max\limits_{1\leq i\leq N}a_i(x)a_{i+1}(x)\leq S_N^*(x).
\end{equation}
Then by \eqref{main thm equa-1} and \eqref{main thm equa-2}, we always have
$$\lim\limits_{N\rightarrow \infty}\frac{S_N(x)-\max_{1\leq i\leq N}a_i(x)a_{i+1}(x)}{N\log^2N}
=\lim\limits_{N\rightarrow \infty}\frac{S_N^*(x)}{N\log^2N} \textmd{ for a.e. } x\in (0,1),$$
which completes the proof.
%\end{proof}

\subsection {Proof of Theorem \ref{convergent in measure}}
%\begin{proof}[Proof of Theorem \ref{convergent in measure}]
For any $\epsilon>0$ and $N\ge1$, we need to estimate $$\mu\left(\left\{x\in(0,1): \left|\frac{S_N(x)}{N\log^2N}-\frac1{2\log2}\right|>\epsilon\right\}\right).$$ For $1\le n\le N$,
let
$$b_n^{**}(x)=
\left\{
              \begin{array}{ll}
            b_n(x), & {\rm if}  \ \  b_n(x)\leq \epsilon N\log^2 N,\\ [2ex]
            0,      & {\rm otherwise},\\
             \end{array}
\right.
$$
and
$$S^{**}_N(x)=\sum\limits_{1\le n\le N}b^{**}_n(x).$$
Take $\psi(N)=\lfloor \epsilon N\log^2 N \rfloor$ and $\varphi(N)=\epsilon N\log^2 N$ respectively in the proof of Lemma \ref{expectation} and Theorem \ref{main thm-measure}. Indeed,
in the proof of Lemma \ref{expectation} and Theorem \ref{main thm-measure}, we have already proved that
\[\mathbb E(S^{**}_N(x))=\frac{N\log^2\psi(N)(1+o(1))}{2\log2}=\frac{N\log^2N(1+o(1))}{2\log2}\]
and
\[\textmd{Var}(S^{**}_N(x))\ll N\varphi(N)\log\varphi(N)\ll\epsilon N^2\log^3 N .\]
Thus by Chebyshev's inequality, we have
\begin{equation}\label{measure-estimate-1}
  \mu\left(\{x: \left|S^{**}_N(x)-\frac{N\log^2N}{2\log2}\right|>\epsilon N\log^2N\}\right)\ll \frac{1}{\epsilon\log N}.
\end{equation}
 On the other hand, for any $1\le n\le N$,  by Lemma \ref{measure of a_1a_2}, we get
 \begin{equation}\label{measure-estimate-2}
   \mu\left(\{x\in(0,1): a_n(x)a_{n+1}(x)\ge \epsilon N\log^2N\}\right)\ll \frac{1}{\epsilon N\log N}.
 \end{equation}
Then by \eqref{measure-estimate-1} and \eqref{measure-estimate-2}, it follows that
 \[\mu\left(\left\{x\in(0,1): \left|\frac{S_N(x)}{N\log^2N}-\frac1{2\log2}\right|>\epsilon\right\}\right) \ll \frac{1}{\epsilon \log N},\]
 which completes the proof.
%\end{proof}

\begin{rem}
  Philipp \cite{Philipp76} proved that for almost all $x\in(0,1)$,
  \[\liminf\limits_{n\to\infty}\frac{\max_{1\le i\le n}a_i(x)}{n/\log\log n}=\frac{1}{\log2}.\]
  Following the proof in \cite{Philipp76} with some modifications of measure estimate by Lemma \ref{measure of a_1a_2},
  we can get for almost all $x\in(0,1)$,
  \[\liminf\limits_{n\to\infty}\frac{\max_{1\le i\le n}a_i(x)a_{i+1}(x)}{n\log n/\log\log n}=\frac{1}{2\log2}.\]
\end{rem}
\section{Proof of Theorem \ref{thm-dimension-1}}\label{proof-for-full-dimension}
%\begin{proof}[Proof of Theorem \ref{thm-dimension-1}]
 Note that the upper bound is trivially $1$, hence we focus on the lower bound.
For any $M\ge 2$, let
\[E_M=\{x\in(0,1):a_n(x)\le M\textmd{ for all } n\ge1\}.\]
It is well known (see \cite{Jarnik1}\footnote{This article is unavailable on the AMS MathSciNet. However, it is available on the The Czech Digital Mathematics Library https://dml.cz/handle/10338.dmlcz/500717}) that
$$\lim\limits_{M\to\infty}\dim_{\mathcal H} E_M=1.$$
It suffices to prove that for any $M\ge 2$, there exists a Cantor subset $E(\phi,M)$ with $ \dim_{\mathcal H} E(\phi,M)=\dim_{\mathcal H} E_M.$
Take $0<\tau<1/2$ such that \begin{equation*}
 \limsup\limits_{n\to\infty}\frac{\log \log \phi(n)}{\log n}<1/2-\tau
\end{equation*} and take $0<\delta<1$ small enough such that
\begin{equation}\label{condition-delta}
  \left(1+\frac{1}{1-\delta}\right)\left(\frac{1}{2}-\tau\right)<1.
\end{equation}
let $\epsilon_k=k^{-\delta}$ for all $k\ge1$.
Now we shall define a sequence of positive integers $n_k$. Let $n_1\ge 3$ be the smallest integer such that
$$ \log \phi(n)<n^{1/2-\tau}$$
for all $n\ge n_1$. For $k\ge 2$, let $n_k$ be the smallest integer such that $n_k\ge n_{k-1}+4$ and
\begin{equation}\label{nk-def}
  \phi(n_k)\ge (1+\epsilon_{k-1})\phi(n_{k-1}).
\end{equation}
Define
\[
\begin{aligned}
 E(\phi,M):=\Big\{x\in(0,1): &a_{n_1}(x)=\lfloor 1/2(1+\epsilon_1)\phi(n_1) \rfloor+1,\\
 & a_{n_{k+1}}(x)=\lfloor 1/2\left((1+\epsilon_{k+1})\phi(n_{k+1})-(1+\epsilon_{k})\phi(n_{k})\right) \rfloor+1,\\
 &a_{n_k-1}(x)=a_{n_k+1}(x)=1\textmd{ for all }k\ge1, 1\le a_i(x)\le M\textmd{ for other }i\Big\}.
\end{aligned}\]
We first prove the following.
 \begin{lem}\label{subset E(phi,M)}
$E(\phi,M)\subset E(\phi)$.
 \end{lem}
 \begin{proof}
 For any $n$ large enough, there exists positive integer $k$ such that $n_k\le n<n_{k+1}$.  For any $x\in E(\phi,M)$, we have $$S_{n_k}(x)\le S_n(x)\le S_{n_{k+1}}(x).$$ Since
$$S_{n_k}(x)\ge \sum\limits_{i=1}^{k}(a_{n_i-1}(x)a_{n_i}(x)+a_{n_i}(x)a_{n_i+1}(x))\ge (1+\epsilon_k)\phi(n_k) $$ and
$$\begin{aligned}
 S_{n_{k+1}}(x)&\le (n_{k+1}-2k)M^2+\sum\limits_{i=1}^{k+1}(a_{n_i-1}(x)a_{n_i}(x)+a_{n_i}(x)a_{n_i+1}(x))\\
&\le (n_{k+1}-2k)M^2+(1+\epsilon_{k+1})\phi(n_{k+1})+2(k+1),
\end{aligned}$$
we have
\begin{equation}\label{sn-estimate}
 (1+\epsilon_k)\phi(n_k)\le S_n(x)\le (n_{k+1}-2k)M^2+(1+\epsilon_{k+1})\phi(n_{k+1})+2(k+1).
\end{equation}
By the definition of $n_k$, we have
either \[\phi(n_k-1)<(1+\epsilon_{k-1})\phi(n_{k-1})\] when $n_k>n_{k-1}+4$ or \[\phi(n_k-4)<(1+\epsilon_{k-1})\phi(n_{k-1})\]
when $n_k=n_{k-1}+4$. It follows that
$$\begin{aligned}
  \frac{\phi(n_k)}{\phi(n_{k-1})}&= \min\left\{\frac{\phi(n_k)}{\phi(n_k-1)}\frac{\phi(n_k-1)}{\phi(n_{k-1})}, \frac{\phi(n_k)}{\phi(n_k-4)}\frac{\phi(n_k-4)}{\phi(n_{k-1})}\right\}\\
  &\le (1+\epsilon_{k-1})\min\left\{\frac{\phi(n_k)}{\phi(n_k-4)},\frac{\phi(n_k)}{\phi(n_k-1)} \right\}.
\end{aligned}$$
Combining the above estimate with  $\lim\limits_{i\to\infty}\frac{\phi(i+1)}{\phi(i)}=1$, we have
\begin{equation}\label{phi-ratio}
  \lim\limits_{k\to\infty}\frac{\phi(n_k)}{\phi(n_{k-1})}=1.
\end{equation}
Since $$\frac{S_{n}(x)}{\phi(n_{k+1})}\le \frac{S_n(x)}{\phi(n)}\le \frac{S_{n}(x)}{\phi(n_k)},$$
we have $$\lim\limits_{n\to\infty}\frac{S_n(x)}{\phi(n)}=1$$
by \eqref{sn-estimate}, \eqref{phi-ratio} and the condition that $\lim\limits_{n\to\infty}\frac{\phi(n)}{n}=\infty$.
That is $x\in E(\phi)$. Thus $$E(\phi,M)\subset E(\phi).$$
\end{proof}

In order to estimate $\dim_{\mathcal H} E(\phi,M)$, we define the map $$f: E(\phi,M)\to E_M$$
by $$f(x)=[a_1(x),\ldots,a_{n_1-2}(x),a_{n_1+2}(x),\ldots,a_{n_k-2}(x),a_{n_k+2}(x),\ldots].$$
This means that if we delete all $a_{n_k-1}(x),a_{n_k}(x),a_{n_k+1}(x)$ from the partial quotients of $x$, then we get all the partial quotients of $f(x)$. For any $n\ge1$, let
\[r(n):=\#\{k: n_k\le n\}.\]
\begin{lem}\label{Lipschitz-map}
  For any $\epsilon>0$, the map $f$ is  $\frac{1}{1+\epsilon}$-Lipschitz.
\end{lem}
\begin{proof}The proof is similar to that in \cite{WuXu} and \cite{LiaoRams} with slight changes.
  Suppose that $x,y\in E(\phi,M)$ with some integer $n$ such that
  \begin{equation}\label{x-y-n}
    a_i(x)=a_i(y)\textmd{ for }1\le i\le n, ~ a_{n+1}(x)\ne a_{n+1}(y).
  \end{equation}
  Then $n+1\not\in\bigcup_{i\ge1}\{n_i-1, n_i, n_{i}+1\}$. It follows that $1\le a_{n+1}(x), a_{n+1}(y)\le M$
and either $I(a_1(x),\ldots,a_n(x),a_{n+1}(x),M+1)$ or $I(a_1(x),\ldots,a_n(x),a_{n+1}(y),M+1)$ is in the gap between $x$ and $y$. By \eqref{length of cylinders} and \eqref{recursive formulae of qn}, there exists a constant $C_M$ only depending on $M$ such that
\[\begin{aligned}
  |I(a_1(x),\ldots,a_n(x),a_{n+1}(x),M+1)|&\ge |I(a_1(x),\ldots,a_n(x),M,M+1)|\\
  &\ge C_M |I(a_1(x),\ldots,a_n(x))|.
\end{aligned}\] Similarly, we also have
\[|I(a_1(x),\ldots,a_n(x),a_{n+1}(y),M+1)|\ge C_M |I(a_1(x),\ldots,a_n(x))|.\] So
\begin{equation}\label{distance by length of cylinder}
 |x-y|\ge C_M |I(a_1(x),\ldots,a_n(x))|.
\end{equation}
On the other hand, $f(I(a_1(x),\ldots,a_n(x)))$ is also a cylinder of order $n-3r(n)$. Let
\[I(b_1,\ldots,b_{n-3r(n)}):=f(I(a_1(x),\ldots,a_n(x))).\]
Note that we get $(b_1,\ldots,b_{n-3r(n)})$ by deleting  $\{a_{n_i-1}(x),a_{n_i}(x),a_{n_i+1}(x) :1\le i\le r(n)\}$ from $(a_1(x),\ldots,a_n(x))$.
We list the following two facts:
\begin{equation}\label{fact 1}
 |I(a_1(x),\ldots,a_n(x))|\ge ca_{i}^{-2}(x)|I(a_1(x),\ldots,a_{i-1}(x),a_{i+1}(x),\ldots,a_n(x))|
\end{equation}
and
\begin{equation}\label{fact 2}
 |I(b_1,b_2,\ldots,b_{n-3r(n)})|\le \tau^{n-3r(n)-1}
\end{equation}
for any $x\in E(\phi,M)$ and any $1\le i\le n$, where $c>0$ and $0<\tau<1$ are two absolute constants. The two facts have been proved and used in \cite{WuXu} (See also \cite[Lemma 4.1]{HuYu} and \cite[Theorem 12]{Khinchin_book64}). Since $a_{n_k}(x)$ is the same for any $x\in E(\phi,M)$ and
any $k\ge1$, we denote $a_{n_k}(x)$ by $a_{n_k}$.
If
\begin{equation}\label{rn and tau}
  \lim\limits_{n\to\infty}\frac{r(n)}{n}=
  \lim\limits_{n\to\infty}\frac{\log (a_{n_1}a_{n_2}\ldots a_{n_{r(n)}})}{n}=0,
\end{equation}
then there exists an integer $n_0\ge1$ only depending on $\epsilon$ such that
\begin{equation}\label{tau and rn}
c^{3r(n)}(a_{n_1}a_{n_2}\ldots a_{n_{r(n)}})^{-2}\ge \tau^{(n-3r(n)-1)\epsilon}
\end{equation}
for all $n\ge n_0$. Thus by \eqref{distance by length of cylinder}, \eqref{fact 1}, \eqref{fact 2} and \eqref{tau and rn},
\[\begin{aligned}
  |x-y|&\ge C_M |I\left(a_1(x),\ldots,\ldots,a_n(x)\right)|\\
  &\ge C_M c^{3r(n)}(a_{n_1}a_{n_2}\ldots a_{n_{r(n)}})^{-2}|I(b_1,b_2,\ldots,b_{n-3r(n)})|\\
  &\ge C_M\tau^{(n-3r(n)-1)\epsilon}|I(b_1,b_2,\ldots,b_{n-3r(n)})|\\
  &\ge C_M |I(b_1,b_2,\ldots,b_{n-3r(n)})|^{1+\epsilon}\\
  &\ge C_M |f(x)-f(y)|^{1+\epsilon}
\end{aligned}\]
for any $n\ge n_0$ and any $x,y\in E(\phi,M)$ satisfying \eqref{x-y-n}, which implies that
$f$ is a $\frac{1}{1+\epsilon}$-Lipschitz map. So we are left with the task of proving \eqref{rn and tau}.
It has been proved  in \cite[eq. (10) and eq. (12)]{WuXu} that if $n_k$ satisfies \eqref{nk-def} then we have
\begin{equation}\label{estimate-of-rn}
 r(n)\ll n^{\frac{1/2-\tau}{1-\delta}}.
\end{equation}
Since
$$\begin{aligned}
  \log (a_{n_1}a_{n_2}\ldots a_{n_{r(n)}})&\ll \log \prod\limits_{k=1}^{r(n)}(1+\epsilon_k)\phi(n_k)\\
  &\ll r(n)\log \phi(n)+(\epsilon_1+\cdots+\epsilon_{r(n)})\\
  &\ll r(n) n^{1/2-\tau}+r(n)^{1-\delta},
  \end{aligned}$$
  it follows that \eqref{rn and tau} holds by \eqref{condition-delta} and \eqref{estimate-of-rn}. This completes the proof.
\end{proof}

By the definition of $E(\phi,M)$ and $f$, the map $f$ is a bijection. From Lemmas \ref{subset E(phi,M)}, \ref{Lipschitz-map} and \cite[Proposition 2.3]{Falconer_book2003}, we have
\[\dim_{\mathcal H} E(\phi)\ge \dim_{\mathcal H} E(\phi,M)\ge \frac{1}{1+\epsilon}\dim_{\mathcal H} E_M.\]
Letting $\epsilon\to 0$ and $M\to \infty$, we get the desired result.

\section{Proof of Theorems \ref{thm-dimension-1/2} and \ref{thm-dimension-1/1+alpha}}\label{proof-of-last-two} For the rest of the paper, we let $\psi(n):=\log \phi(n)$ which is monotonically increasing and $\lim_{n\to\infty}\psi(n)=\infty$.
\subsection{The lower bounds}
To prove the lower bound, we shall use the following lemma  from Falconer's book  \cite[Example 4.6]{Falconer_book2003}.
\begin{lem}\label{Falconer's lemma}
Let $E_0=[0,1]$ and let $E_n$ be a finite union of disjoint closed intervals with $E_{n}\subset E_{n-1}$ for any $n\ge1$.  Suppose that
each interval of $E_{n-1}$ contains at least $m_n(\ge 2)$ intervals
of $E_{n}$  and the intervals of $E_n$ are separated by gaps at least $\epsilon_n$ with $0<\epsilon_{n+1}<\epsilon_n$ for all $n\ge1$. Let $E=\bigcap\limits_{n\ge0}E_n$. Then
$$\dim_{\mathcal H} E\ge \liminf\limits_{n\to\infty}\frac{\log(m_1\cdots m_{n-1})}{-\log(m_n\epsilon_n)}.$$
\end{lem}

The next lemma gives a lower bound of $\dim_{\mathcal H} E(\phi)$ in Theorem  \ref{thm-dimension-1/2}.
\begin{lem}\label{lower-bound-1/2} Let $x_n=\psi(n)-\psi(n-1)$ for $n\ge 2$.
  Suppose that $x_n$ is monotonically decreasing with
  \begin{equation}\label{conditions-on-psi-n}
    \lim\limits_{n\to\infty}x_n=\lim\limits_{n\to\infty}\frac{x
    _n-x_{n-1}} {x_n^2}=0.
  \end{equation}
Then $\dim_{\mathcal H} E(\phi)\ge 1/2$.
\end{lem}
\begin{proof}
Let $d_1,d_2,\ldots$ be positive real numbers defined by $d_1=1$, $d_2=\phi(1)$ and
\begin{equation}\label{def-of-bn}
  d_nd_{n+1}=\phi(n)-\phi(n-1).
\end{equation}
for $n\ge 2$.
 Let $g(t)=\log\frac{e^t-1}{t}$ for $t>0$. Then
\[\lim\limits_{t\to 0}g'(t)=1/2.\]
Thus we have
\[g(x_n)-g(x_{n-1})=(x_n-x_{n-1})\left(\frac{1}{2}+o(1)\right)\]
and hence
\begin{equation}\label{estimate-by-g}
  \frac{e^{x_n}-1}{e^{x_{n-1}}-1}=\frac{x_n}{x_{n-1}}e^{\frac{x_n-x_{n-1}}{2}(1+o(1))}.
\end{equation}
Then by the successive use of \eqref{estimate-by-g}, we have
$$
\begin{aligned}
  \frac{d_{n+1}}{d_{n-1}}&=\frac{d_nd_{n+1}}{d_{n-1}d_n}=\frac{\phi(n)-\phi(n-1)}{\phi(n-1)-\phi(n-2)}\\ &=\frac{e^{\psi(n)}-e^{\psi(n-1)}}{e^{\psi(n-1)}-e^{\psi(n-2)}}
\\ &  =\frac{e^{\psi(n-1)}(e^{x_n}-1)}{e^{\psi(n-2)}(e^{x_{n-1}}-1)}\\
  &=e^{x_{n-1}}\frac{x_n}{x_{n-1}}e^{\frac{x_n-x_{n-1}}{2}(1+o(1))}\\ &=\frac{x_n}{x_{n-1}}e^{\frac{x_n+x_{n-1}}{2}(1+o(1))}\\
  &=e^{\frac{x_n-x_{n-1}}{x_{n}}(1+o(1))}e^{\frac{x_n+x_{n-1}}{2}(1+o(1))}.%\\ &=e^{\frac{x_n+x_{n-1}}{2}(1+o(1))}\\ &=e^{\frac{\psi(n)-\psi(n-2)}{2}(1+o(1))}.
\end{aligned}$$
From \eqref{conditions-on-psi-n}, it follows that $  \frac{x_n-x_{n-1}}{x_n}= o(\frac{x_n+x_{n-1}}{2})$. Combining this estimate with the above estimate, we get

$$
\begin{aligned}
  \frac{d_{n+1}}{d_{n-1}}&=e^{\frac{x_n-x_{n-1}}{x_{n}}(1+o(1))}e^{\frac{x_n+x_{n-1}}{2}(1+o(1))}\\ &=e^{\frac{x_n+x_{n-1}}{2}(1+o(1))}=e^{\frac{\psi(n)-\psi(n-2)}{2}(1+o(1))}.
\end{aligned}$$
So we have
\begin{equation}\label{estimate of bn/bn-1}
  \frac{d_n}{d_{n-2}}=e^{\frac{\psi(n-1)-\psi(n-3)}{2}(1+o(1))}.
\end{equation}
If $n$ is even, then by \eqref{estimate of bn/bn-1} and \eqref{conditions-on-psi-n},
$$
  d_n=\frac{d_n}{d_{n-2}}\cdot \frac{d_{n-2}}{d_{n-4}}\cdots \frac{d_4}{d_2}d_2= e^{\frac{\psi(n-1)}{2}(1+o(1))}=e^{\frac{\psi(n)}{2}(1+o(1))}.
$$
Similarly, if $n$ is odd, we still have
\begin{equation}\label{estimate of bn}
    d_n=e^{\frac{\psi(n)}{2}(1+o(1))}.
\end{equation}
Take an integer $N\ge1$ large enough such that
\begin{equation}\label{N-def-large}
d_n\ge2,  \ \ \frac{d_n}{\psi(n)}\ge 2
\end{equation}
for all $n\ge N$.
Let
\[E=\left\{x\in[0,1]: a_n(x)=1\textmd{ for }n\le N, d_{n}\le a_{n}(x)\le \left(1+\frac{1}{\psi(n)}\right)d_{n}\textmd{ for }n>N \right\}.\]
Then by the definition of $d_n$, we have
\[E\subset E(\phi).\]
For any $n\ge N$ and any positive integers $a_1,\ldots,a_n$, let
\[J(a_1,\ldots,a_n):=\textmd{cl} \bigcup\limits_{a_{n+1}}I(a_1,\ldots,a_n,a_{n+1}),\]
where the union is taken over all integers $a_{n+1}$ such that
\[d_{n+1}\le a_{n+1}\le \left(1+\frac{1}{\psi(n+1)}\right)d_{n+1}.\]
and $\textmd{cl}$ stands for the closure of a set in $\mathbb{R}$.
 Let
$a_1=a_2\cdots=a_N=1$ and let
\[E_n=\bigcup\limits_{a_{N+1},\ldots,a_{N+n}}J(a_1,a_2,\ldots,a_{N+n})\]
for $n\ge1$, where the union is taken over all integers $a_{N+1},\ldots,a_{N+n}$ such that
\[d_{N+i}\le a_{N+i}\le \left(1+\frac{1}{\psi(N+i)}\right)d_{N+i}\]
for all $1\le i\le n$. Then
\[E=\bigcap\limits_{n\ge1}E_n.\]
Set
\[m_n:=\#\left\{k\in \mathbb{Z}: d_{N+n}\le k\le \left(1+\frac{1}{\psi(N+n)}\right)d_{N+n}\right\}.\]
Then each interval of $E_{n-1}$ contains $m_n$ disjoint intervals of $E_n$.
By  \eqref{estimate of bn} and \eqref{N-def-large}, we have
\begin{equation}\label{estimate of mn}
2\le m_n=e^{\frac{\psi(N+n)}{2}(1+o(1))}.
\end{equation}
Note that for any two adjacent intervals $J(a_1,\ldots,a_{N+n-1},d)$ and $J(a_1,\ldots,a_{N+n-1},d+1)$ of $E_n$,
either $I(a_1,\ldots,a_{N+n-1},d,1)$ or $I(a_1,\ldots,a_{N+n-1},d+1,1)$ is contained in the gap between the two adjacent intervals.
Write
\[\theta_n:=\min\limits_{a_{N+1},\ldots,a_{N+n}}|I(a_1,\cdots,a_{N+n},1)|,\]
where the minimum is taken over all integers $a_{N+1},\ldots,a_{N+n}$ with
\[d_{N+i}\le a_{N+i}\le \left(1+\frac{1}{\psi(N+i)}\right)d_{N+i}\]
for all $1\le i\le n$.
By \eqref{estimate of length of cylinder by partial quotients}, we have \[|I(a_1,\cdots,a_{N+n},1)|\ge 2^{-2(N+n+2)}(a_{N+1}\cdots a_{N+n})^{-2}.\]
Thus, by \eqref{estimate of bn}, we obtain
\begin{equation}\label{epsilon-n}
 \theta_n\ge \epsilon_n:=2^{-2(N+n+2)}\prod\limits_{k=N+1}^{N+n}\left(\left(1+\frac{1}{\psi(N+i)}\right)d_{N+i}\right)^{-2}=e^{-\sum_{k=N+1}^{N+n}\psi(k)(1+o(1))}.
\end{equation}
This means that the disjoint intervals of $E_n$ are separated by gaps of at least $\epsilon_n$.
Then by Lemma \ref{Falconer's lemma} combined with the estimates \eqref{conditions-on-psi-n}, \eqref{estimate of mn} and \eqref{epsilon-n},
$$\begin{aligned}
  \dim_{\mathcal H} E&\ge \liminf\limits_{n\to\infty}\frac{m_1\cdots m_{n-1}}{-\log(\epsilon_n m_n)}\\ &=\liminf\limits_{n\to\infty}\frac{\frac{1}{2}\sum_{k=N+1}^{N+n-1}\psi(k)(1+o(1))}{-\frac{\psi(N+n)}{2}(1+o(1))
  +\sum_{k=N+1}^{N+n}\psi(k)(1+o(1))}\\
  &=\liminf\limits_{n\to\infty}\frac{\sum_{k=1}^{n-1}\psi(k)}{
  \psi(n)+2\sum_{k=1}^{n-1}\psi(k)}\\ &=1/2,
\end{aligned}
$$
which completes the proof.
\end{proof}

The next lemma gives the lower bound for Theorem \ref{lower-bound-1/1+alpha}.

\begin{lem}\label{lower-bound-1/1+alpha}
Suppose that $\psi(n+1)-\psi(n)$ is monotonically increasing with
\begin{equation}\label{difference-ratio}
  \lim\limits_{n\to\infty}\frac{\psi(n+1)-\psi(n)}{\psi(n)-\psi(n-1)}=\alpha.
\end{equation}
Then  $\dim_{\mathcal H} E(\phi)\ge \frac{1}{1+\alpha}$.
\end{lem}
\begin{proof}
Since $\psi(n+1)-\psi(n)$ is monotonically increasing, the limit  $c:=\lim\limits_{n\to\infty}\psi(n+1)-\psi(n)$ exists.  Let $d_n$ be defined as in \eqref{def-of-bn}.
 We distinguish two cases.
\smallskip

\noindent {\em Case 1.} $0<c<\infty.$ Then $\psi(n)=cn(1+o(1))$ and $\alpha=1$. Then
\[\begin{aligned}
  \frac{d_{n+1}}{d_{n-1}}&=\frac{d_nd_{n+1}}{d_{n-1}d_n}\\ &=\frac{e^{\psi(n)}-e^{\psi(n-1)}}{e^{\psi(n-1)}-e^{\psi(n-2)}}
 \\ & =\frac{e^{\psi(n)}(1-e^{\psi(n-1)-\psi(n)})}{e^{\psi(n-1)}(1-e^{\psi(n-2)-\psi(n-1)})}\\
  &=e^{\psi(n)-\psi(n-1)+o(1)}\\ &=e^{c+o(1)},
\end{aligned}\]
which implies that
\[d_n=e^{\frac{cn}{2}(1+o(1))}=e^{\frac{\psi(n)}{2}(1+o(1))}=e^{\frac{\psi(n)}{1+\alpha}(1+o(1))}.\]

\smallskip

\noindent {\em Case 2.} $c=\infty.$ Then we have
\[\begin{aligned}
  \frac{d_{n+1}}{d_{n-1}}&=\frac{e^{\psi(n)}(1-e^{\psi(n-1)-\psi(n)})}{e^{\psi(n-1)}(1-e^{\psi(n-2)-\psi(n-1)})}
  &=e^{\psi(n)-\psi(n-1)+O(1)}.
\end{aligned}\]
If $n$ is even, then
\[\begin{aligned}
  d_n&=\frac{d_n}{d_{n-2}}\cdot \frac{d_{n-2}}{d_{n-4}}\cdots \frac{d_4}{d_2}d_2\\ &= d_2 e^{\psi(n-1)-\psi(n-2)+\cdots+\psi(3)-\psi(2)+O(n)}\\ &=e^{\frac{\psi(n)}{1+\alpha}(1+o(1))}
\end{aligned}\]
by \eqref{difference-ratio}. If $n$ is odd, we can similarly get
  $$d_n=e^{\frac{\psi(n)}{1+\alpha}(1+o(1))}.$$
So in both cases, we get the same estimate for $d_n$. As in the proof of Lemma \ref{lower-bound-1/2}, we similarly define the Cantor subset  $E$ of $E(\phi)$ by $d_n$ and $\psi(n)$. Then
$$\dim_{\mathcal H} E\ge  \liminf\limits_{n\to\infty}\frac{\sum_{k=1}^{n-1}\psi(n)}{
  \psi(n)+2\sum_{k=1}^{n-1}\psi(k)}=\frac{1}{1+\alpha}.$$

\end{proof}

\subsection{The upper bounds}
\begin{lem}\label{numbers-estimate}
For any positive integer $n\ge2$, let
$$\delta(n):=\#\{(a,b)\in \mathbb{N}\times\mathbb{N}:ab=n\}.$$
For any $\epsilon>0$, there exists a constant $c_{\epsilon}$ depending on $\epsilon$ such that $$\delta(n)\leq c_{\epsilon} n^{\epsilon}$$
for all integers $n\ge2$.
\end{lem}

\begin{proof}
Let $p_i$ be the $i$-th prime number, i.e.,
$$p_1=2, p_2=3, p_3=5, \ldots.$$
Let $M\geq 1$ be the smallest integer such that
$$p_M^{\epsilon}\geq 2$$
and $l_0$ be the smallest integer such that
$$2^{\epsilon l}\geq l+1$$ for all $l\geq l_0$.
Write
$$n=p_{i_1}^{k_1}p_{i_2}^{k_2}\cdots p_{i_m}^{k_m}$$
for some positive integers $i_1<i_2<\cdots <i_m$. Then
$$\delta(n)=(k_1+1)(k_2+1)\cdots(k_m+1).$$
If $i_j> M$, we have
\[k_j+1\leq 2^{k_j}\le p_{M}^{\epsilon k_j}\leq p_{i_j}^{\epsilon k_j}.\]
If $i_j\leq M$, we distinguish two cases.
\begin{enumerate}
    \item[(i)] $k_j\leq l_0$, then $k_j+1\leq l_0+1\leq 2^{ l_0}\le 2^{ l_0}p_{i_j}^{\epsilon k_j}$.
    \item[(ii)] $k_j> l_0$, then $k_j+1\leq 2^{\epsilon k_j}\leq p_{i_j}^{\epsilon k_j}\le 2^{ l_0} p_{i_j}^{\epsilon k_j}$.
\end{enumerate}
Now we divide $p_{i_1},p_{i_2},\ldots, p_{i_m}$ into two parts by $p_M$ as follows
$$p_{i_1}<\cdots<p_{i_{q}}\le p_{M}, \ p_M<p_{i_{q+1}}<\cdots<p_{i_m}.$$
Then we have $q\le M$,
$$\prod\limits_{j=1}^{q}(k_j+1)\leq \prod\limits_{j=1}^{q}2^{ l_0}p_{i_j}^{\epsilon k_j}
\leq 2^{M l_0}\prod\limits_{j=1}^{q}p_{i_j}^{\epsilon k_j}$$
and
$$\prod\limits_{j=q+1}^{m}(k_j+1)\leq \prod\limits_{j=q+1}^{m}p_{i_j}^{\epsilon k_j}.$$
It follows that
$$\delta(n)=\prod\limits_{j=1}^{m}(k_j+1)\leq c\prod\limits_{j=1}^{m}p_{i_j}^{\epsilon k_j}=cn^{\epsilon},$$
where $c=2^{Ml_0}$.
\end{proof}

We quote Lemma 2.1 from \cite{LiaoRams} that we will use in proving Lemma \ref{upper-bound-1/2-0} below.

\begin{lem}\label{upper-estimate-product-sum}
 For any $s\in(1/2,1)$, for all $m\ge n\ge1$, we have
 \[\sum\limits_{(i_1,\ldots,i_n)\in\Gamma_n(m)}\prod\limits_{k=1}^n i_k^{-2s}\le \left(\frac{9}{2}(2+\zeta(2s))\right)^n m^{-2s}.\]
where $$\Gamma_n(m):=\{(i_1,\ldots,i_n)\in \mathbb{Z}_{+}^n: i_1+\cdots+i_n=m\}.$$
\end{lem}

\begin{lem}\label{upper-bound-1/2-0}
 Suppose that for any positive number $M>1$, there exists a subsequence $\{n_k\}$ of $\mathbb{N}$ such that
 \begin{equation}\label{condition-difference1}
 \liminf\limits_{k\to\infty}\psi(n_k)-\psi(n_{k-1}+1)>0,
 \end{equation}and
  \begin{equation}\label{condition-difference2}
n_k-n_{k-1}\ge2, \ \frac{\psi(n_k)}{n_k-n_{k-1}}\ge M
 \end{equation}
 holds for all $k$ large enough. Then we have
 \[\dim_{\mathcal H} E(\phi)\le 1/2.\]
\end{lem}

\begin{proof}
It suffices to prove that for any $s>1/2$,
\[\dim_{\mathcal H} E(\phi)\le s.\]
Take $\epsilon\in(0,2s-1)$ and let
\begin{equation}\label{def-of-M}
  M=(2s-1-\epsilon)^{-1}\log \left(c_{\epsilon} \frac{9}{2}(2+\zeta(2s-\epsilon))\right),
\end{equation}
where $c_{\epsilon}$ is defined in Lemma \ref{numbers-estimate}. Then for $M$ defined by \eqref{def-of-M}, there exists $\{n_k\}$
such that \eqref{condition-difference1} and \eqref{condition-difference2}  hold.
Take $\delta>0$ such that
\[\psi(n_k)-\psi(n_{k-1}+1)\ge \delta\]
for all $k$ large enough and take
$0<\alpha<1$ such that
\[\frac{1+\alpha}{1-\alpha}<e^{\delta}.\]
For any $x\in E(\phi)$, we have
\[(1-\alpha)\le\frac{S_n(x)}{e^{\psi(n)}}\le(1+\alpha)\]
for all $n$ large enough. It follows that
\[\begin{aligned}(1-\alpha)e^{\psi(n_{k+1})}-(1+\alpha)e^{\psi(n_k+1)}&\le S_{n_{k+1}}(x)-S_{n_k+1}(x)\\ &\le (1+\alpha)e^{\psi(n_{k+1})}-(1-\alpha)e^{\psi(n_k+1)}\end{aligned}\] for all $k$ large enough.
Let $c_1=(1-\alpha)-(1+\alpha)e^{-\delta}$ and $c_2=1+\alpha$. Then
\begin{equation}\label{block-sum-estimate}
 c_1e^{\psi(n_{k+1})}\le \sum\limits_{j=n_k+2}^{n_{k+1}}a_j(x)a_{j+1}(x)\le c_2e^{\psi(n_{k+1})}
\end{equation}
for all $k$ large enough. Take $L\ge1$ such that \eqref{condition-difference1} and \eqref{condition-difference2} hold for all $k\ge L$.
Then for any $K\ge L$,
\[E(\phi)\subset \bigcup\limits_{K\ge L}~\bigcup\limits_{a_1,\ldots,a_{N_K+1}\ge1}I(a_1,\ldots,a_{N_K+1})\cap F_K(\phi),\]
where
\[F_K(\phi)=\{x\in[0,1]: \eqref{block-sum-estimate}\textmd{ holds for all }k\ge K\}.\]
It suffices to prove that for any $K\ge L$ and any positive integers $a_1,\ldots,a_{N_K+1}$,
\[\dim_{\mathcal H} I(a_1,\ldots,a_{N_K+1})\cap F_K(\phi)\le s.\]
For any $i\ge K$, let
\[A_i=\left\{(a_{n_i+2},\ldots,a_{n_{i+1}+1})\in\mathbb{N}^{n_{i+1}-n_{i}}: c_1e^{\psi(n_{i+1})}\le \sum\limits_{n=n_i+2}^{n_{i+1}}a_na_{n+1}\le c_2e^{\psi(n_{i+1})}\right\}.\]
Then for any $k\ge K$,
\[\{I(a_1,a_2,\ldots,a_{n_{k+1}+1}): (a_{n_i+2},\ldots,a_{n_{i+1}+1})\in A_i\textmd{ for all }K\le i\le k\}\]
is a cover of $I(a_1,\ldots,a_{N_K+1})\cap F_K(\phi)$. Then by \eqref{estimate of length of cylinder by partial quotients}, the $s$-dimensional Hausdorff measure $\mathcal H^s$ can be estimated as
$$\begin{aligned}
  \mathcal{H}^s\left(I(a_1,\ldots,a_{N_K+1})\cap F_K(\phi)\right)&\le \liminf\limits_{k\to\infty}\sum\limits_{(a_{n_i+2},\ldots,a_{n_{i+1}+1})\in A_i, K\le i\le k}|I(a_1,a_2,\ldots,a_{n_{k+1}+1})|^s\\
  &\le\liminf\limits_{k\to\infty}\prod\limits_{i=K}^k\sum\limits_{(a_{n_i+2},\ldots,a_{n_{i+1}+1})\in A_i}(a_{n_i+2}\cdots a_{n_{i+1}+1})^{-2s}\\
  &=\liminf\limits_{k\to\infty}\prod\limits_{i=K}^k\Lambda_i(s),
\end{aligned}$$
where
\[\Lambda_i(s):=\sum\limits_{(a_{n_i+2},\ldots,a_{n_{i+1}+1})\in A_i}(a_{n_i+2}\cdots a_{n_{i+1}+1})^{-2s}.\]
Next we shall estimate $\Lambda_i(s)$.
We divide the integers $n_{i}+2,n_i+3,\ldots,n_{i+1}$ into two parts:
\[I_{i, 0}:=\left\{n_i+2k: k\in \mathbb{Z}, 1\le k\le \frac{n_{i+1}-n_i}{2}\right\}\]
and
\[I_{i, 1}:=\left\{n_i+2k+1: k\in \mathbb{Z}, 1\le k\le \frac{n_{i+1}-n_i-1}{2}\right\}.\]
If $(a_{n_i+2},\ldots,a_{n_{i+1}+1})\in A_i$, then either
\begin{equation*}\label{sum-0-i}
  \frac{c_1}{2}e^{\psi(n_{i+1})}\le\sum\limits_{j\in I_{i, 0}}a_ja_{j+1}\le c_2e^{\psi(n_{i+1})}
\end{equation*}
or
\begin{equation}\label{sum-1-i}
  \frac{c_1}{2}e^{\psi(n_{i+1})}\le\sum\limits_{j\in I_{i, 1}}a_ja_{j+1}\le c_2e^{\psi(n_{i+1})}.
\end{equation}
We consider the case that $n_{i+1}-n_i-1$ is odd and \eqref{sum-1-i} holds. The proof of other cases is similar.
In this case,
\begin{equation}\label{I-1-i-estimate}
  \# I_{i, 1}=\frac{n_{i+1}-n_i}{2}-1\le \frac{n_{i+1}-n_i}{2}.
\end{equation}
Let $b_j=a_ja_{j+1}$ and $$r_j=\# \{(x,y)\in \mathbb{N}^2: xy=b_j\}.$$ Then \eqref{sum-1-i} implies that
\begin{equation}\label{sum-bj-estimate}
 \frac{c_1}{2}e^{\psi(n_{i+1})}\le \sum\limits_{j\in I_{i, 1}}b_j\le c_2e^{\psi(n_{i+1})}.
\end{equation}
Let $c_s=\sum\limits_{n\ge1}n^{-2s}$.
Then $$\prod\limits_{j\in I_{i, 1}}b_j=a_{n_i+3}\cdots a_{n_{i+1}}$$ and
$$\begin{aligned}
  \Lambda_i(s)&\le \sum\limits_{a_{n_i+2}\ge 1}\sum\limits_{a_{n_{i+1}+1}\ge1}a_{n_i+2}^{-2s}a_{n_{i+1}+1}^{-2s}\sum\limits_{(a_{n_i+2},\ldots,a_{n_{i+1}+1})\in A_i}(a_{n_i+3}\cdots a_{n_{i+1}})^{-2s}\\
  &\le c_s^2 \sum\limits_{(a_{n_i+2},\ldots,a_{n_{i+1}+1})\in A_i}(a_{n_i+3}\cdots a_{n_{i+1}})^{-2s}\\ &\le c_s^2\sum \prod\limits_{j\in I_{i, 1}} r_jb_j^{-2s},
\end{aligned}$$
where the last sum is taken over all $(b_j)_{j\in I_{i, 1}}$ such that \eqref{sum-bj-estimate} holds.
By Lemma \ref{numbers-estimate}, we have $r_j\le c_{\epsilon} b_j^{\epsilon}$ and hence by \eqref{I-1-i-estimate}, we have
\[\prod\limits_{j\in I_{i, 1}} r_jb_j^{-2s}\le c_{\epsilon}^{\frac{n_{i+1}-n_i}{2}}\prod\limits_{j\in I_{i, 1}}b_j^{-2s+\epsilon}.\]
Let
\[\mathcal{T}_i(m)=\left\{(b_j)_{j\in I_{i, 1}}:\sum\limits_{j\in I_{i, 1}}b_j=m\right\}\] and
\[D_i=\left\{m\in \mathbb{Z}:\frac{c_1}{2}e^{\psi(n_{i+1})} \le m\le c_2e^{\psi(n_{i+1})}\right\}.\]
Then
$$\begin{aligned}
\Lambda_i(s)&\le c_s^2c_{\epsilon}^{\frac{n_{i+1}-n_i}{2}}\sum\limits_{m\in D_i}~~
\sum\limits_{(b_j)_{j\in I_{i, 1}}\in \mathcal{T}_i(m)}~~\prod\limits_{j\in I_{i, 1}}b_j^{-2s+\epsilon}\\
&\le c_s^2c_{\epsilon}^{\frac{n_{i+1}-n_i}{2}}\sum\limits_{m\in D_i}\left(\frac{9}{2}(2+\zeta(2s-\epsilon))\right)^{\frac{n_{i+1}-n_i}{2}}m^{-2s+\epsilon},
\end{aligned}$$
where we have used Lemma \ref{upper-estimate-product-sum} and \eqref{I-1-i-estimate} in the second inequality.
Since $\# D_i\le c_3 e^{\psi(n_{i+1})}$ with $c_3=c_2-\frac{c_1}{2}+1$ and $\min\limits_{m\in D_i}\ge \frac{c_1}{2}e^{\psi(n_{i+1})}$, we have
$$\begin{aligned}
\Lambda_i(s)&\le c_s^2c_{\epsilon}^{\frac{n_{i+1}-n_i}{2}}\left(\frac{9}{2}(2+\zeta(2s-\epsilon))\right)^{\frac{n_{i+1}-n_i}{2}}
\left(\frac{c_1}{2}e^{\psi(n_{i+1})}\right)^{-2s+\epsilon}c_3 e^{\psi(n_{i+1})}\\
&=Ce^{-(2s-\epsilon-1)\psi(n_{i+1})+\frac{n_{i+1}-n_i}{2}\log c_4},
\end{aligned}$$
where $C=c_s^2c_3(c_1/2)^{-2s+\epsilon}$
 and $c_4=c_{\epsilon}\frac{9}{2}(2+\zeta(2s-\epsilon))$ are independent of $i$. By \eqref{condition-difference1}, \eqref{condition-difference2}  and \eqref{def-of-M}, we have
 $$(2s-\epsilon-1)\psi(n_{i+1})\ge(n_{i+1}-n_i)\log c_4$$ for any $i\ge K$ and hence
 \[\Lambda_i(s)\le C e^{-\frac{2s-\epsilon-1}{2}\psi(n_{i+1})}.\]
For other cases, we can similarly get the above estimates for $\Lambda_i(s)$. Thus
\[\mathcal{H}^s\left(I(a_1,\ldots,a_{N_K+1})\cap F_K(\phi)\right)\le \liminf\limits_{k\to\infty}\prod_{i=K}^k Ce^{-\frac{2s-\epsilon-1}{2}\psi(n_{i+1})}=0.\]
So, by the definition of Hausdorff dimension,  $\dim_{\mathcal H} I(a_1,\ldots,a_{N_K+1})\cap F_K(\phi)\le s$, which completes the proof.
\end{proof}

%\begin{lem}\label{upper-bound-1/2-non0}
%Suppose that
%\end{lem}

Finally, we are now in a position to prove Theorem 1.8 and Theorem 1.7. We will use the following well-known lemma from \cite{Luczak}.
\begin{lem}\label{Luczak-estimate}
For any $b>1$, the set
\[\{x\in(0,1): a_n(x)\ge e^{b^n}\textmd{ for infinitely many }n\}\]
has Hausdorff dimension $1/(1+b)$.
\end{lem}

\begin{proof}[Proof of Theorem \ref{thm-dimension-1/2}]
Take $n_k=\lfloor k^{1/\alpha}\rfloor$ for $1/2<\alpha<1$ and $n_k=2k$ for $\alpha\ge1$. If $\alpha=1/2$, take $$n_k=\lfloor \frac{k^2}{(3M)^2}\rfloor.$$ Then by Lemma \ref{upper-bound-1/2-0}, we have
\[\dim_{\mathcal H} E(\phi)\le 1/2.\]
On the other hand, if $1/2\le \alpha<1$, then we get
$\dim_{\mathcal H} E(\phi)\ge 1/2$ by Lemma \ref{lower-bound-1/2}. If $\alpha\ge 1$, then
we  get $\dim_{\mathcal H} E(\phi)\ge 1/2$ by Lemma \ref{lower-bound-1/1+alpha} with $\alpha=1$.
\end{proof}

\begin{proof}[Proof of Theorem \ref{thm-dimension-1/1+alpha}]
By Lemma \ref{lower-bound-1/1+alpha}, we have $\dim_{\mathcal H} E(\phi)\ge \frac{1}{1+\alpha}$. On the other hand, for any $1<b<\alpha$, we have
\[E(\phi)\subset \{x\in(0,1): a_n(x)\ge e^{b^n}\textmd{\textmd{ for infinitely many }}n\}.\]
Thus by Lemma \ref{Luczak-estimate}, $\dim_{\mathcal H} E(\phi)\le \frac{1}{1+b}$ for any $1<b<\alpha$. So, $\dim_{\mathcal H} E(\phi)\le \frac{1}{1+\alpha}$.
\end{proof}

%\bibliographystyle{amsplain}
%%
%\bibliography{../bibliography4}

\begin{thebibliography}{10}


\bibitem{Aaronson}
Jon Aaronson, \emph{On the ergodic theory of non-integrable functions and infinite measure spaces,} Israel J. Math. 27 (1977), no. 2, 163--173.  \MR{0444899}

\bibitem{BBH1}
Ayreena {Bakhtawar}, Philip {Bos}, and Mumtaz {Hussain}, \emph{{The sets of
  Dirichlet non-improvable numbers vs well-approximable numbers}},  Ergodic Theory Dynam. Systems \textbf{40} (2020), no. 12, 3217--3235. \MR{4170601}

\bibitem{BBH2}
Ayreena Bakhtawar, Philip Bos, and Mumtaz Hussain, \emph{Hausdorff dimension of
  an exceptional set in the theory of continued fractions}, Nonlinearity
  \textbf{33} (2020), no.~6, 2615--2639. \MR{4105371}

\bibitem{Billingsley}
Patrick Billingsley, \emph{Ergodic theory and information}, John Wiley \& Sons,
  Inc., New York-London-Sydney, 1965. \MR{0192027}


  \bibitem{BHS}
Philip Bos, Mumtaz Hussain, and David Simmons, \emph{The generalised Hausdorff measure of sets of Dirichlet non-improvable numbers}, Preprint: arXiv:2010.14760.

\bibitem{DiamondVaaler}
Harold~G. Diamond and Jeffrey~D. Vaaler, \emph{Estimates for partial sums of
  continued fraction partial quotients}, Pacific J. Math. \textbf{122} (1986),
  no.~1, 73--82. \MR{825224}

\bibitem{Falconer_book2003}
Kenneth Falconer, \emph{Fractal geometry}, second ed., John Wiley \& Sons,
  Inc., Hoboken, NJ, 2003, Mathematical foundations and applications.
  \MR{2118797}


\bibitem{GHPZ}
Stefano Galatolo, Mark  Holland, Tomas  Persson, and Yiwei Zhang, \emph{Anomalous time-scaling of extreme events in infinite systems and Birkhoff sums of infinite observables},  Discrete Contin. Dyn. Syst. 41 (2021), no. 4, 1799--1841. \MR{4211203}
%\bibitem{Falconer_book2013}
%\bysame, \emph{Fractal geometry}, third ed., John Wiley \& Sons, Ltd.,
%  Chichester, 2014, Mathematical foundations and applications. \MR{3236784}

\bibitem{FanJoranLiaoRams} Ai-hua Fan, Thomas Jordan, Lingmin Liao, Micha\l \ Rams,  \emph{ Multifractal analysis for expanding interval maps with infinitely many branches}, Trans. Amer. Math. Soc. \textbf{367} (2015), no. 3, 1847-1870. \MR{3286501}
\bibitem{HuYu}Hui Hu, Yueli Yu, \emph{On Schmidt's game and the set of points with non-dense orbits under a class of expanding maps}, J. Math. Anal. Appl. \textbf{418} (2014), no.~2, 906-920. \MR{3206688}

\bibitem{HuWu}
Lingling Huang and Jun Wu, \emph{Uniformly non-improvable {D}irichlet set via
  continued fractions}, Proc. Amer. Math. Soc. \textbf{147} (2019), no.~11,
  4617--4624. \MR{4011499}

  \bibitem{HuWuXu}
Lingling Huang, Jun Wu, and Jian Xu, \emph{Metric properties of the product of
  the partial quotients in continued fractions}, Israel J. Math. \textbf{238} (2020), no. 2, 901--943. \MR{4145821}



\bibitem{HKWW}
Mumtaz Hussain, Dmitry Kleinbock, Nick Wadleigh, and Bao-Wei Wang,
  \emph{Hausdorff measure of sets of {D}irichlet non-improvable numbers},
  Mathematika \textbf{64} (2018), no.~2, 502--518. \MR{3798609}

\bibitem{Iommi-Jordan}
Godofredo Iommi and Thomas Jordan, \emph{Multifractal analysis of {B}irkhoff
  averages for countable {M}arkov maps}, Ergodic Theory Dynam. Systems
  \textbf{35} (2015), no.~8, 2559--2586. \MR{3456607}

\bibitem{Jarnik1}
Vojt{\v e}ch Jarn\'ik, \emph{Zur metrischen {T}heorie der diophantischen
  {A}pproximationen}, Prace mat. fiz. \textbf{36} (1928-29), 91--106 (German).

  \bibitem{KS1}
Marc,   Kesseb\"ohmer, Tanja I. Schindler, \emph{ Mean convergence for intermediately trimmed Birkhoff sums of observables with regularly varying tails,} Nonlinearity 33 (2020), no. 10, 5543--5566. \MR{4151417}

  \bibitem{KS2}
Marc,   Kesseb\"ohmer, Tanja I. Schindler, \emph{ Intermediately trimmed strong laws for Birkhoff sums on subshifts of finite type}, Dyn. Syst. 35 (2020), no. 2, 275--305.\MR{4120831}

\bibitem{Khinchin_book64}
Aleksandr Khinchin, \emph{Continued fractions}, The University of Chicago
  Press, Chicago, Ill.-London, 1964. \MR1451873

\bibitem{Khinchin1935}
Aleksandr Khinchin, \emph{Metrische {K}ettenbruchprobleme}, Compositio Math.
  \textbf{1} (1935), 361--382. \MR{1556899}

\bibitem{KleinbockWadleigh}
Dmitry Kleinbock and Nick Wadleigh, \emph{A zero-one law for improvements to
  {D}irichlet's {T}heorem}, Proc. Amer. Math. Soc. \textbf{146} (2018), no.~5,
  1833--1844. \MR{3767339}

\bibitem{LiaoRams}
Lingmin Liao and Micha\l  \ Rams, \emph{Subexponentially increasing sums of
  partial quotients in continued fraction expansions}, Math. Proc. Cambridge
  Philos. Soc. \textbf{160} (2016), no.~3, 401--412. \MR{3479541}

\bibitem{Luczak}
Tomasz {\L}uczak, \emph{On the fractional dimension of sets of continued
  fractions}, Mathematika \textbf{44} (1997), no.~1, 50--53. \MR{1464375}

\bibitem{Philipp76} Walter Philipp, \emph{A conjecture of Erd{\"o}s on continued fractions}, Acta Arith. \textbf{28}  (1975/1976), no.~4, 379--386. \MR{0387226}

\bibitem{Philipp88}
Walter Philipp, \emph{Limit theorems for sums of partial quotients of continued
  fractions}, Monatsh. Math. \textbf{105} (1988), no.~3, 195--206. \MR{939942}

\bibitem{WuXu}
Jun Wu and Jian Xu, \emph{On the distribution for sums of partial quotients in
  continued fraction expansions}, Nonlinearity \textbf{24} (2011), no.~4,
  1177--1187. \MR{2776116}

\bibitem{Xu08}
Jian Xu, \emph{On sums of partial quotients in continued fraction expansions},
  Nonlinearity \textbf{21} (2008), no.~9, 2113--2120. \MR{2430664}

\end{thebibliography}

\providecommand{\bysame}{\leavevmode\hbox to3em{\hrulefill}\thinspace}
\providecommand{\MR}{\relax\ifhmode\unskip\space\fi MR }
\providecommand{\MRhref}[2]{%
  \href{http://www.ams.org/mathscinet-getitem?mr=#1}{#2}
}
\providecommand{\href}[2]{#2}

 \end{document}